\documentclass[a4paper, 10pt, twoside, notitlepage]{amsart}

\usepackage[utf8]{inputenc}
\usepackage{color}
\usepackage{amsmath}
\usepackage{amssymb}
\usepackage{amsthm}
\usepackage{graphicx}
\usepackage{esint}
\usepackage{geometry}
\usepackage[colorlinks=true,linkcolor=blue]{hyperref}
\usepackage{verbatim}

\theoremstyle{plain}
\newtheorem{thm}{Theorem}
\newtheorem{prop}{Proposition}[section]
\newtheorem{lem}[prop]{Lemma}
\newtheorem{cor}[prop]{Corollary}

\theoremstyle{definition}
\newtheorem{defi}[prop]{Definition}
\newtheorem{rmk}[prop]{Remark}

\def\be{\begin{equation}}
\def\ee{\end{equation}}

\newcommand {\R} {\mathbb{R}} \newcommand {\Z} {\mathbb{Z}}
 \newcommand {\N} {\mathbb{N}}
 \newcommand {\Sf} {\mathbb{S}^1}
\newcommand {\dH} {d\mathcal{H}}
\newcommand {\p} {\partial}

\newcommand{\argmin}{\operatorname{argmin}}
\newcommand{\mres}{\mathbin{\vrule height 1.6ex depth 0pt width
0.13ex\vrule height 0.13ex depth 0pt width 1.3ex}}

\DeclareMathOperator{\vol}{vol}
\DeclareMathOperator{\di}{div}
\DeclareMathOperator{\grad}{grad}

\usepackage{xspace}
\title{Uniformly compressing mean curvature flow}

\author[W.~Shi]{Wenhui Shi}
\address[W.~Shi]{CMUC, Department of
Mathematics, University of Coimbra, 3001-501 Coimbra, Portugal}{}
\email{wshi@mat.uc.pt}

\author[D.~Vorotnikov]{Dmitry Vorotnikov}
\address[D.~Vorotnikov]{CMUC, Department of
Mathematics, University of Coimbra, 3001-501 Coimbra, Portugal}{}
\email{mitvorot@mat.uc.pt}

\begin{document}
\begin{abstract} Michor and Mumford showed that the mean curvature flow is a gradient flow on a  Riemannian structure with a degenerate geodesic distance. It is also known to destroy the uniform density of gridpoints on the evolving surfaces. We introduce a related geometric flow which is free of these drawbacks. Our flow can be viewed as a formal gradient flow on a certain submanifold of the Wasserstein space of probability measures endowed with Otto’s Riemannian structure. We obtain a number of analytic results concerning well-posedness and long-time stability which are however restricted to the 1D case of evolution of loops.
\end{abstract}
\maketitle


\vspace{10pt}


\section{Introduction}

The mean curvature flow \cite{Sm12,Cold}
is a geometric flow which describes the behaviour of a $k$-dimensional submanifold $M_t\subset \R^d$, $1\leq k<d$, which evolves over time $t$ according to the law \be\label{mcf1} \frac {dx}{dt}= \vec H(x),\ee where $x$ is an arbitrary point of $M_t$, and $\vec H(x)$ is the mean curvature vector of the submanifold in $x$. It has huge variety of applications ranging from formation of grain boundaries in metals to image processing. The mean curvature flow (MCF) is the formal negative gradient flow of the volume functional $\vol\simeq\mathcal H^{k}$, \be \partial_t M_t = -\grad  \vol(M_t),\ee where the ``manifold of $k$-dimensional submanifolds'' is equipped with the $L^2$ Riemannian structure (see \cite{BBM14,MM5} or our formula \eqref{eq:Rm}), and the gradient is understood in the sense of this structure. Hence, as usual in the context of gradient flows \cite{villani03topics,AGS06}, the volume functional, being the driving entropy of the gradient flow, decays with time in the following way:
\be \label{decay} \frac {d}{dt} \vol(M_t) = -\langle\grad \vol(M_t),\grad \vol(M_t)\rangle_{M_t}\\=-\int_{M_t}|\vec H|^2\, d \mathcal H^{k}. \ee

There is an avalanche of works about the theoretical and numerical aspects of the mean curvature flow. One can introduce minimizing movement schemes \`a la de Giorgi which exploit the gradient flow structure \cite{Alm93,LS95}.  The solutions should a priori collapse in finite time, but singular behaviour before the breakdown does not permit existence of smooth solutions up to the final moment. There exist various strategies to go beyond the first singularity. The pioneering work \cite{B78} relaxed the notion of solution so that the evolving objects were barely varifolds.  The level-set approach \cite{OS88} led to solutions in the Crandall-Lions viscosity sense \cite{ES91,CGG91,AS96}. For the curve-shortening flow, i.e., when $k=1$, one can define and construct a weak solution as a limit of certain curves which live in an ambient space of higher dimension and which are called ramps \cite{AG92,Per03}. Other types of weak solutions for the curve-shortening flow were introduced and studied in \cite{A91i,A91ii,D97,H15,BD18}.

There are two issues which mar the overall harmony.  The first one is that the underlying Riemannian distance is degenerate (i.e., one can connect any two surfaces with a path of arbitrarily small Riemannian length), cf. \cite{MM5,MM06,MM07,BHM12,BBM14}, which is unpleasant since a non-degenerate metric space structure is an important precondition in treatment of gradient flows, see \cite{AGS06,San17}.
The second issue is that the Hausdorff measure $\mathcal H^{k}$ is not uniformly contracted by the flow, that is, if $T_t:x(0)\mapsto x(t)$ is the flow operator describing the trajectories of material particles forming $M_t$ in the ambient space $\R^d$, then the property \be \label{good} (T_t)_\#\left( \frac 1 {\vol(M_0)} \mathcal H^{k}\mres M_0 \right)=\frac 1 {\vol(M_t)}\mathcal H^{k}\mres M_t\ee is violated except for some very special scenarios as a shrinking sphere. From the numerical perspective, this means that the flow destroys the uniform density of gridpoints on $M_t$, which is unwelcome and may cause computational instabilities \cite{MS01,SC09}. For the curve-shortening flow in the plane this can be fixed \cite{AL86,MS01} by adding a certain tangential motion to the right-hand side of \eqref{mcf1} in order to conserve the uniform density of the moving particles without affecting the evolution of the shapes.

In this paper, we suggest a different approach which simultaneously eliminates the two above mentioned drawbacks of the mean curvature flow, and which is applicable for any dimensions $k$ and $d$. The idea is to consider the flow which is the closest possible to the original MCF \eqref{mcf1} in a certain least-squares sense among the flows which uniformly contract the $k$-Hausdorff measure (in other words, which respect the uniform density of gridpoints). We employ the infinite-dimensional manifold $\mathcal A_k$ (consisting of objects of the form $\eta=\lambda \xi$, where $\lambda>0$ is a scalar which quantifies the volume of $M$, and $\xi$ is a volume-preserving immersion in the sense of \cite{GS79}) endowed with the parametrization-invariant $L^2$ metric. Our flow is driven by the orthogonal projection of the mean curvature onto $T\mathcal A_k$. We dub the resulting object the uniformly compressing mean curvature flow (UCMCF) because the evolving surfaces can be thought of as being constituted by fluid particles whose density depends merely on time (the surfaces in question up to a time-dependent constant are \emph{incompressible membranes} \cite{Ev73,Fol94,Mol17}). The UCMCF is by construction the negative gradient flow of the volume functional on $\mathcal A_k$. It is a genuinely geometric flow in the sense that the evolution of submanifolds $M_t=\eta(t)(M)$ does not depend on their parametrization.

Unlike the tangentially-corrected MCF \cite{AL86,MS01}, our flow differs from the classical MCF in the normal direction and thus the geometrical evolution of the submanifolds along the two flows do not coincide in general (although they do coincide for a shrinking sphere). Nevertheless, we show that the qualitative behavior of UCMCF is quite similar to the one of the usual MCF, thus it may be used as a substitute for the MCF in applications.

We will observe that the our flow collapses in finite time, and in order to study the evolution and stability of the shapes before the breakdown we need to renormalize the problem both in time and in space.
Surprisingly enough, our normalized flow is also a gradient flow: namely, the positive gradient flow of the $L^2$-mass on the space of volume-preserving immersions. Our recent work \cite{ShV17} studies the gradient flow of a different functional (potential energy) on a similar Riemannian structure, which turns out to be a model for an overdamped fall of an inextensible string in a gravitational field. A similar mechanical interpretation for our normalized UCMCF is an overdamped motion of an inextensible loop ($k=1$) or an incompressible membrane ($k>1$) repelled from the origin with the force field identically equal to the radius vector.

Other gradient flows of inextensible strings were considered in \cite{Koi96,Oel11,Oel14,O07,O08}. In those papers, additional forth-order terms coming from the bending energy appear, which help to secure non-degenerate parabolicity of the equations and decrease the difficulties created by the Lagrange multiplier, cf. \eqref{eq:constrxi}, \eqref{eq:xi0} below.

Both the original and normalized flows can be viewed as formal gradient flows on certain submanifolds of the Wasserstein space \cite{otto01,villani03topics,villani08oldnew} of probability measures endowed with Otto's Riemannian structure, cf. Section \ref{opt.transport}.

We will mostly work with the immersed curves $1=k<d$. The Appendix is devoted to the general case  $1\leq k<d$. It turns out to be more convenient to analyze the normalized flow, which allows  us to descry the asymptotic behaviour of the shapes near the breakdown. We show local strong well-posedness of the problem. We characterize the steady states, and prove global existence of strong solutions for the initial data which are close to the steady states which maximize the driving $L^2$-mass, i.e., to simple circles.  We establish the exponential decay of such a global solution to a steady state. We address the global solvability for any Lipschitz initial curve which does not need to be close to the equilibria, and we prove existence of suitably defined weak solutions. Our approach is based on approximation of the gradient flow on the manifold of volume-preserving immersions by Hilbertian gradient flows; in particular, we do not use the ramps. Unlike \cite{AG92,Per03,A91i,A91ii,D97,H15}, our weak solutions are $H^1$-regular in time.

In this paper, we make a technical and geometrically non-restrictive assumption that the center of mass is fixed at the origin; otherwise the center of mass would fly away to infinity.

\section{The flow}
\label{sec:grad}
\subsection{Uniformly compressing curve-shortening flow} \label{1st}
Let $\mathbb{S}^1\simeq \R/\Z$ denote the circle of length $1$. For $d\geq 2$ let $\mathcal{K}$ be the space of closed curves $\mathcal{K}:=\{\eta: \eta\in H^2(\mathbb{S}^1;\R^d), \int_0^1\eta(s) ds =0\}$.
Let $L:\mathcal{K}\rightarrow \R$, $L(\eta):=\int_0^1|\p_s\eta|\ ds$ be the length functional. We consider the space of immersed curves with the constant speed parametrization, i.e.,
\begin{align*}
 \mathcal{A}:=\{\eta\in \mathcal{K}: \ &|\p_s\eta(s)|=L(\eta)>0 \text{ for all } s\in \Sf\}.
 \end{align*}
Arguing as Theorem A.1 in \cite{Preston-2012} we see that $\mathcal{A}$ is a smooth Hilbert submanifold of $\mathcal{K}$, with the tangent space
 \begin{align*}
 T_\eta\mathcal{A}=&\{w\in H^2(\mathbb{S}^1;\R^d): \frac{d}{ds}\left(\p_s w(s)\cdot \p_s \eta(s)\right)=0 \text{ for a.e. } s\in \Sf\\
 & \text{ and }\int_0^1w(s)\ ds=0\},\ \eta\in \mathcal{A}.
 \end{align*}
We endow $\mathcal{K}$ with the Riemannian metric
\begin{align*}
\langle v,w\rangle_{T_\eta\mathcal{K}}:=\int_0^1 v(s)\cdot w(s)|\p_s\eta(s)|\ ds,
\end{align*}
which is invariant under reparametrization, cf. \cite{MM06}. It induces a metric on $\mathcal{A}$:
\begin{align}\label{eq:metric}
\langle v,w\rangle_{T_\eta\mathcal{A}}=\int_0^1 v(s)\cdot w(s) L(\eta)\ ds.
\end{align}

\begin{prop} \label{rimnond1} The Riemannian distance $d_{\mathcal{A}}$ is non-degenerate. \end{prop}

\begin{proof} Take any two closed curves $\eta_0,\eta_1\in \mathcal{A}$. Renormalizing if needed, we can suppose that $\|\eta_0-\eta_1\|_{L^2(\mathcal S^1)}=1$. We claim that \be \label{deg_1} d_{\mathcal{A}}(\eta_0,\eta_1)\geq m:=\frac 1 2 \min(L(\eta_0),\|\eta_0\|^2_{L^2(\mathcal S^1)},2).\ee If not, there exists a $C^1$ path $\eta:[0,1]\to \mathcal{A}$, $\eta(0)=\eta_0$, $\eta(1)=\eta_1$, with Riemannian length
\begin{equation} \label{deg_2}
\begin{split}
 \mathcal L(\eta)&:=\int_0^1\sqrt{\langle \dot\eta(t),\dot\eta(t)\rangle_{T_{\eta(t)}\mathcal{A}}} \,dt\\
&=\int_0^1 \|\dot\eta(t)\|_{L^2(\mathcal S^1)} \sqrt{L(\eta(t))}\,dt< m.
\end{split}
\end{equation}
 Since by the Minkowski inequality and integration by parts
 \begin{align*}
 \int_0^1\|\dot\eta(t)\|_{L^2(\mathcal S^1)}\, dt \geq \|\eta_1-\eta_0\|_{L^2(\mathbb{S}^1)}= 1,
 \end{align*}
 we have $L(\eta(t))<m^2$ for some $t$. Due to continuity of $L(\eta(t))$ (and recall that $L(\eta_0)\geq 2m$ from \eqref{deg_1}), there exists $t_*\in (0,1)$ such that $L(\eta(t_*))=m$ and $L(\eta(t))>m$ for $t<t_*$. Then \be \label{deg_3} \mathcal L (\eta)\geq \mathcal L (\eta|_{(0,t_*)})\geq \sqrt m\|\eta_0-\eta(t_*)\|_{L^2(\mathcal S^1)}.\ee By Wirtinger's inequality, \be \label{deg_4} \|\eta(t_*)\|_{L^2(\mathcal S^1)}\leq \frac 1 {2\pi} \|\p_s\eta(t_*)\|_{L^2(\mathcal S^1)}=\frac m {2\pi}.\ee Combining  \eqref{deg_2}---\eqref{deg_4}, we infer $$\sqrt{2m}\leq \|\eta_0\|_{L^2(\mathcal S^1)}\leq \|\eta_0-\eta(t_*)\|_{L^2(\mathcal S^1)}+\|\eta(t_*)\|_{L^2(\mathcal S^1)}\leq \sqrt m\left(1 +\frac 1 {2\pi}\right),$$  which is a contradiction. \end{proof}

We are interested in the formal gradient flow of the length functional $L(\eta)=\int_0^1|\p_s\eta| \ ds$, $\eta\in\mathcal{A}$, under the metric \eqref{eq:metric}: \begin{equation} \label{eqgf} \p_t\eta=-\grad_{\mathcal A} L(\eta(t)).\end{equation}
In order to derive the PDE formulation of \eqref{eqgf}, we compute (formally) the orthogonal projection from $T_\eta\mathcal{K}$ (which can be identified with $\mathcal{K}$) onto $T_\eta\mathcal{A}$ with respect to the metric \eqref{eq:metric} (cf. Proposition 3.2 in \cite{Preston-2012}).

\begin{lem}\label{lem:proj}
Let $\eta\in \mathcal{A}\cap H^4(\mathbb{S}^1;\R^d)$. The orthogonal projection $P_\eta: T_\eta\mathcal{K}\rightarrow T_\eta\mathcal{A}$ is given by
\begin{align}
P_\eta(z)&=z-\p_s(\sigma\p_s\eta),\text{ where } \sigma: \mathbb S^1\to \R \text{ solves }\notag\\
L^2\p_{ss}\sigma -|\p_{ss} \eta |^2\sigma &= \p_sz\cdot \p_s\eta+const, \label{eq:constraint}\\
\int_0^1\sigma(s) ds &=0. \label{eq:sigma}
\end{align}
\end{lem}
\begin{proof}
\begin{itemize}
\item[(1)] We first show that for any $\sigma$ satisfying \eqref{eq:sigma}, the vector field $\p_s(\sigma\p_s\eta)$ is orthogonal to $T_\eta\mathcal{A}$. Indeed, given any $w\in T_\eta\mathcal{A}$,
\begin{align*}
\langle \p_s(\sigma\p_s\eta), w\rangle_{T_\eta\mathcal{K}} &=L(\eta)\int_0^1 \p_s(\sigma\p_s\eta) \cdot w \ ds\\
&=L(\eta)\sigma\p_s\eta\cdot w \big|_{s=0}^{s=1}-L(\eta)\int_0^1\sigma\p_s\eta\cdot \p_s w\ ds\\
&=0.
\end{align*}
The first term is zero since we work on $\Sf$. The second term is zero because $\p_s\eta\cdot \p_sw=const$ and $\sigma$ has mean zero.
\item[(2)] Next we show $w=z-\p_s(\sigma\p_s\eta)\in T_\eta\mathcal{A}$. We will mainly check the condition $\p_sw\cdot \p_s\eta=const$. Indeed,
\begin{align*}
\p_sw\cdot \p_s\eta=\p_sz\cdot \p_s\eta-\p_{ss}\sigma|\p_s\eta|^2-2\p_s\sigma\p_{ss}\eta\cdot\p_s\eta-\sigma \p_{sss}\eta\cdot \p_s\eta.
\end{align*}
The constant speed parametrization $|\p_s\eta|=L=const$ yields $\p_{ss}\eta\cdot \p_s\eta=0$ and $\p_{sss}\eta\cdot\p_s\eta=-|\p_{ss}\eta|^2$. Thus
\begin{align*}
\p_sw\cdot \p_s\eta=\p_s z\cdot \p_s\eta-\p_{ss}\sigma L^2 +\sigma |\p_{ss}\eta|^2,
\end{align*}
which is constant by \eqref{eq:constraint}.
\end{itemize}
\end{proof}

Since $\mathcal A$ is a smooth submanifold of the space $\mathcal K$, the general definition of the gradient \cite{Lang} implies that \begin{align*}\grad_{\mathcal A}L(\eta)=P_\eta(\grad_{\mathcal K}L(\eta)).\end{align*}
Standard calculus of variations shows that the first $L^2$-variation of $L(\eta)$ is
\begin{align*}
\delta L(\eta)=-\p_s\left(\frac{\p_s\eta}{|\p_s\eta|}\right).
\end{align*}
But\begin{align*}\langle \grad_{\mathcal K}L(\eta),w\rangle_{T_\eta\mathcal{K}}=\int_0^1 \delta L(\eta)(s)\cdot w(s)\ ds,\end{align*} for every $w\in T_\eta\mathcal{K}\simeq\mathcal K$.
We conclude that
\begin{align*}
\grad_{\mathcal K}L(\eta)=-\frac{1}{|\p_s\eta|}\p_s\left(\frac{\p_s\eta}{|\p_s\eta|}\right).
\end{align*}
By Lemma \ref{lem:proj} the orthogonal projection of the negative gradient in $\mathcal K$ to the tangent space $T_\eta\mathcal{A}$ is
\begin{align*}
P_\eta(-\grad_{\mathcal K} L(\eta))=\frac{\p_{ss}\eta-\p_s(\sigma\p_s\eta)}{L^2},
\end{align*}
where $\sigma:\Sf\to \R$ satisfies
\begin{align*}
 L^2\p_{ss}\sigma-\sigma|\p_{ss}\eta|^2=\p_{sss}\eta\cdot \p_s\eta+const=-|\p_{ss}\eta|^2+const, \quad \int_0^1\sigma ds =0.
\end{align*}
To determine the constant, we integrate in $s$ and thus obtain $const=\int_0^1(1-\sigma)|\p_{ss}\eta|^2ds$. Letting $\tilde{\sigma}:=1-\sigma$ we get the expression for the gradient flow \eqref{eqgf}:
\begin{equation}\label{eq:grad}
\p_t\eta(t,s)=L(t)^{-2}\p_s(\tilde{\sigma}(t,s)\p_s\eta(t,s)),
\end{equation}
where $L(t):=L(\eta(t))$, and the Lagrange multiplier $\tilde{\sigma}(t,s)$ satisfies
\begin{equation*}
\begin{split}
&\p_{ss}\tilde{\sigma}(t,s)-L(t)^{-2}\tilde{\sigma}(t,s)|\p_{ss}\eta(t,s)|^2=-L(t)^{-2}\int_0^1\tilde{\sigma}(t,s)|\p_{ss}\eta(t,s)|^2 ds\text{ for all }(t,s) \\
&\text{and } \int_0^1\tilde{\sigma}(t,s)\ ds =1\text{ for all } t.
\end{split}
\end{equation*}

\begin{rmk}[Sign of $\tilde{\sigma}$]
	\label{rmk:sign_sigma}
	From
	$\|\p_{ss}\eta-P_\eta(\p_{ss}\eta)\|_{L^2(\mathbb{S}^1)}\leq \|\p_{ss}\eta\|_{L^2(\mathbb{S}^1)}$ as well as the orthogonality $\p_s\eta\cdot \p_{ss}\eta=0$ everywhere it follows that
	$$\|\p_{ss}\eta\|_{L^2(\mathbb{S}^1)}^2\geq \|\p_s\sigma\|_{L^2(\mathbb{S}^1)}^2+\|\sigma\p_{ss}\eta\|_{L^2(\mathbb{S}^1)}^2.$$
	Using $\tilde{\sigma}=1-\sigma$ we can rewrite the above inequality as
	\begin{align*}
	2\int_{\mathbb{S}^1}\tilde{\sigma}|\p_{ss}\eta|^2 ds\geq \int_{\mathbb{S}^1}\tilde{\sigma}^2|\p_{ss}\eta|^2 ds + \int_{\mathbb{S}^1}(\p_s\sigma)^2 ds.
	\end{align*}
	Thus $\int_{\mathbb{S}^1}\tilde{\sigma}|\p_{ss}\eta|^2 ds \geq 0$ and the equality holds if and only if $\tilde{\sigma}\equiv 0$ on $\mathbb{S}^1$, which cannot happen since the mean of $\tilde \sigma$ is $1$. Then by \eqref{eq:grad} $\tilde{\sigma}$ satisfies an inhomogeneous elliptic equation with a negative inhomogeneity, and provided $\eta(t)\in C^2(\Sf)$ by the strong maximum principle we infer that $\tilde{\sigma}>0$.
\end{rmk}

\begin{rmk}\label{rmk:constraint}
Assume $\eta(t,s)$ is a classical solution to equation \eqref{eq:grad} emanating from $\eta_0$ which satisfies the constraints $|\p_s\eta_0(s)|\equiv L(\eta_0)$ and $\int_0^1\eta_0(s)ds=0$. Then the constraints are preserved along the flow, i.e.,
\begin{align*}
|\p_s\eta(t,s)|\equiv L(t)\text{ and }\int_0^1\eta(t,s)ds=0 \text{ for each }t.
\end{align*}
To show the first equality we let $Z(t,s):=|\p_s\eta(t,s)|^2-L(t)^2$.  Then a direct computation shows that $Z$ satisfies
\begin{align*}
\p_tZ&=L^{-2}\tilde{\sigma}\p_{ss}Z+2L^{-2}\p_s\tilde{\sigma}\p_sZ+2\left(\p_{ss}\tilde{\sigma}-L^{-2}\tilde{\sigma}|\p_{ss}\eta|^2-L\p_tL\right)\\
&=L^{-2}\tilde{\sigma}\p_{ss}Z+2L^{-2}\p_s\tilde{\sigma}\p_sZ\quad \text{ with }Z(0,s)=0.
\end{align*}
Here we have used the equation of $\tilde{\sigma}$ (note that by (i) of Proposition \ref{prop:length} below, the right hand side of the equation of $\tilde{\sigma}$ is equal to $L\p_tL$).
By Remark \ref{rmk:sign_sigma} above $\tilde{\sigma}(t,s)>0$ for the classical solutions $\eta$. Thus $Z(t,s)=0$ by the strong maximum principle for parabolic equations. To show that $\int_0^1\eta=0$ is preserved along the flow, we integrate the equation along $\Sf$ to obtain $\p_t\int_0^1\eta(t,s)ds =0$.
\end{rmk}

\subsection{Evolution of the length and the $L^2$-mass}

We now exploit the gradient flow structure to derive some evolution properties of the length and $L^2$-mass of the curves. Throughout this subsection we assume the existence of the $H^2$ solutions  to the gradient flow \eqref{eq:grad}.

The first proposition is about the evolution of the length functional.
\begin{prop}\label{prop:length}
Let $\eta$ be a solution to \eqref{eq:grad}. Then
\begin{itemize}
\item [(i)] $\p_tL(\eta)=-L^{-3}\int_0^1\tilde{\sigma}|\p_{ss}\eta|^2 ds$.
\item [(ii)] $\p_{tt}L(\eta)^2 \geq 0$.
\end{itemize}
\end{prop}
\begin{proof}
1. Since $|\p_s\eta(s,t)|=L(t)$, then
\begin{align*}
\p_t L&=\p_t |\p_s\eta(t,s)|\\
&= \frac{\p_s\eta(t,s)}{|\p_s\eta(t,s)|}\cdot \p_{ts}\eta(t,s) \\
&=\frac{\p_s\eta(t,s)}{L}\cdot \p_s\left(\frac{\tilde{\sigma}\p_{ss}\eta}{L^2}+ \frac{\p_s\tilde{\sigma}\p_s\eta}{L^2}\right).
\end{align*}
Note that the LHS does not depend on $s$. Thus an integration in $s$ from $0$ to $1$ and an integration by parts yield the desired equality. \\

2. Let $\lambda(t):=-L(t)\p_tL(t)=L^{-2}\int_0^1\tilde{\sigma}|\p_{ss}\eta|^2ds$. It suffices to show that $\p_t\lambda(t)\leq 0$. Note that one can rewrite the equation of $\tilde{\sigma}$ as
\begin{align*}
\p_{ss}\tilde{\sigma}-L^{-2}\tilde{\sigma}|\p_{ss}\eta|^2 = -\lambda.
\end{align*}
Let $N$ be the unit inner normal to the curve $\eta$, and let $k$ be such that $\p_{ss}\eta=kN$. Differentiation of the equation of $\tilde{\sigma}$ in time yields
\begin{align*}
\p_{ss}\dot{\tilde{\sigma}}-\dot{\tilde{\sigma}}k^2 L^{-2}-2\tilde{\sigma} k\dot{k} L^{-2} + 2L^{-3}\p_tL \tilde{\sigma}k^2 =-\dot{\lambda}.
\end{align*}
We multiply the above equation by $\tilde{\sigma}$ and integrate in $s$. Integrating by parts, we get
\begin{align*}
\dot{\lambda}&=-\int_0^1 \dot{\tilde{\sigma}}\p_{ss}\tilde{\sigma} ds +L^{-2}\int_0^1 \dot{\tilde{\sigma}}\tilde{\sigma} k^2 ds +2L^{-2}\int_0^1 \tilde{\sigma}^2k\dot{k} ds-2L^{-3}\p_tL\int_0^1\tilde{\sigma}^2k^2 ds\\
&=\lambda\int_0^1\dot{\tilde{\sigma}}  ds +2L^{-2}\int_0^1 \tilde{\sigma}^2k\dot{k} ds-2L^{-3}\p_tL\int_0^1\tilde{\sigma}^2k^2 ds\\
&=2L^{-2}\int_0^1 \tilde{\sigma}^2k\dot{k} ds-2L^{-3}\p_tL\int_0^1\tilde{\sigma}^2k^2 ds.
\end{align*}
Here in the second equality we have used the equation of $\tilde{\sigma}$. The last equality is due to $\int_0^1\tilde{\sigma}(t,s) ds\equiv 1$.
Now we compute $k\dot{k}$. Differentiating the equation of $\eta$ in $s$ we infer
\begin{align*}
\p_{ts}\eta = L^{-2}\left(\p_{ss}\tilde{\sigma} \p_s\eta + 2\p_{s}\tilde{\sigma}\p_{ss}\eta +\tilde{\sigma} \p_{sss}\eta\right).
\end{align*}
Using that $\p_{sss}\eta\cdot \p_s\eta=-k^2$ and $\p_{sss}\eta\cdot N=\p_sk$ we get
\begin{align*}
L^2\p_{t s}\eta &=\left(\p_{ss}\tilde{\sigma} -L^{-2}k^2 \tilde{\sigma}\right)\p_s\eta+\left(2\p_s\tilde{\sigma}k+\tilde{\sigma}\p_sk\right)N+\tilde{\sigma} R\\
&=-\lambda\p_s\eta+\left(2\p_s\tilde{\sigma}k+\tilde{\sigma}\p_sk\right)N+\tilde{\sigma}R
\end{align*}
where $R:=\p_{sss}\eta-(L^{-2}\p_{sss}\eta\cdot \p_s\eta)\p_s\eta-(\p_{sss}\eta\cdot N)N$ and we have used the equation of $\tilde{\sigma}$. Since $\p_sN\cdot N=0$ and $N\cdot R=0$, we have
\begin{align*}
L^2\p_{t ss}\eta \cdot \p_{ss}\eta =-\lambda k^2+k\p_s\left(2\p_s\tilde{\sigma}k+\tilde{\sigma}\p_sk\right)+\tilde{\sigma}(\p_sR\cdot \p_{ss}\eta).
\end{align*}
But $R\cdot \p_{ss}\eta=0$, hence $\p_sR\cdot \p_{ss}\eta= -R\cdot \p_{sss}\eta=-|R|^2$.
Thus
\begin{align} \label{exfl}
L^2\p_{t ss}\eta \cdot \p_{ss}\eta=L^2 k\dot{k}=-\lambda k^2+k\p_s\left(2\p_s\tilde{\sigma}k+\tilde{\sigma}\p_sk\right)-\tilde{\sigma}|R|^2.
\end{align} Until this moment we implicitly assumed that $\p_{ss} \eta\neq 0$. However, \eqref{exfl} is still valid in the points with $\p_{ss} \eta= 0$ since we can make an agreement that $k=| R|=0$ in those points.
Plugging the expression \eqref{exfl} into the equation for $\dot{\lambda}$ we derive
\begin{align*}
\dot{\lambda}&=-2L^{-4}(\lambda+L\p_tL)\int_0^1\tilde{\sigma}^2k^2 ds \\
&+ 2L^{-4}\int_0^1\tilde{\sigma}^2k\p_s(2\p_s\tilde{\sigma}k+\tilde{\sigma}\p_sk)ds - 2L^{-4}\int_0^1\tilde{\sigma}^3|R|^2 ds.
\end{align*}
The first term on the right hand side vanishes due to the definition of $\lambda$. Then an integration by parts yields
\begin{equation}\label{eq:dot_lambda}
\begin{split}
L^4\dot{\lambda}&=-2\int_0^1 \p_s(\tilde{\sigma}^2 k )(\tilde{\sigma}\p_s k)ds - 4\int_0^1\p_s(\tilde{\sigma}^2k)\p_s\tilde{\sigma} k ds -2 \int_0^1\tilde{\sigma}^3|R|^2 ds\\
&=-2\int_0^1\tilde{\sigma}\left(\tilde{\sigma}\p_s k + 2\p_s\tilde{\sigma} k\right)^2 ds-2 \int_0^1\tilde{\sigma}^3|R|^2 ds.
\end{split}
\end{equation}
Note that $\tilde{\sigma}\geq 0$ by Remark \ref{rmk:sign_sigma}. Thus $\dot{\lambda}\leq 0$ and we complete the proof for (ii).
\end{proof}

In the next proposition we show that the $L^2$-mass of the solution decays with constant speed.

\begin{prop}\label{prop:L2}
Let $\eta$ be a solution to the gradient flow \eqref{eq:grad}. Let $M(t):=\frac{1}{2}\int_0^1|\eta(t,s)|^2 ds$ be the $L^2$-mass. Then $\p_tM(t)=-1$.
\end{prop}
\begin{proof}
We multiply the equation of $\eta$ in \eqref{eq:grad} by $\eta$ and integrate in $s$. An integration by parts implies
\begin{align*}
\p_tM(t)=-L^{-2}\int_0^1 \tilde{\sigma}|\p_s\eta|^2 ds.
\end{align*}
Using that $|\p_s\eta|=L$ and $\int_0^1\tilde{\sigma} ds=1$ we obtain the desired equality.
\end{proof}

An immediate consequence of Proposition \ref{prop:L2} is that the flow becomes extinct in finite time.

\begin{cor}\label{cor:ext}
$M(t)\rightarrow 0$ as $t\rightarrow t^\ast$, where $t^\ast=M(0)=\frac{1}{2}\int_0^1|\eta_0(s)|^2 ds$.
\end{cor}

It is also possible to obtain the decay rate of $L(t)$ near the extinction time $t^\ast$.

\begin{cor}\label{cor:ext2}
Let $L_0:=L(\eta_0)>0$ and let $t^\ast $ be the extinction time as in Corollary \ref{cor:ext}. Then for all $t\in [0,t^\ast)$,
\begin{itemize}
\item[(i)] $2\sqrt{2}\pi\sqrt{t^\ast-t}\leq L(t)\leq L_0\sqrt{\frac{t^\ast-t}{t^\ast}}$,
\item[(ii)] $4\pi^2\leq -L(t)\p_tL(t)\leq -(L\p_tL)(t)\big|_{t=0}$.
\end{itemize}
\end{cor}
\begin{proof}
1. The lower bound follows directly from the Wirtinger's inequality. Indeed, by Proposition \ref{prop:L2}
$$M(t)=t^\ast-t\text{ for } t\in [0,t^\ast).$$
On the other hand, by Wirtinger's inequality
\begin{align*}
M(t)\leq \frac{1}{2}\frac{1}{4\pi^2}\int_0^1|\p_s\eta|^2ds=\frac{L(t)^2}{8\pi^2}.
\end{align*}
Here we have used the assumption $\int_0^1\eta=0$. Combining the above two inequalities together we obtain the lower bound $L(t)\geq 2\sqrt{2}\pi\sqrt{t^\ast-t}$.

Next we show the upper bound. Using $\p_tM(t)=\int_0^1\eta\cdot \p_t\eta =-1$ and the H\"older's inequality we deduce
\begin{align}\label{eq:dtL0}
\int_0^1|\p_t\eta|^2 ds \geq \frac{1}{\int_0^1|\eta|^2}=\frac{1}{2M(t)}=\frac{1}{2(t^\ast-t)}.
\end{align}
On the other hand, by the gradient flow structure
\begin{align}\label{eq:dtL}
\p_tL= -L\int_0^1|\p_t\eta|^2 ds.
\end{align}
Thus $-\p_t\ln L\geq \frac{1}{2(t^\ast-t)}$. Then an integration in $t$ from $0$ to $t$ yields
\begin{align*}
L(t)\leq L_0\sqrt{\frac{t^\ast-t}{t^\ast}}, \quad t\in [0,t^\ast).
\end{align*}

2. By (ii) of Proposition \ref{prop:length}, $t\mapsto -L\p_t L$ is monotone decreasing. Thus $-L\p_tL \leq -L\p_tL \big|_{t=0}$. To see the lower bound, we note that by \eqref{eq:dtL},\eqref{eq:dtL0} and Wirtinger's inequality
\begin{align*}
-L(t)\p_tL(t)=L(t)^2\int_0^1|\p_t\eta|^2 \geq \frac{L(t)^2}{2M(t)}\geq 4\pi^2 .
\end{align*}
\end{proof}

\section{Normalized flow}\label{sec:normalize_flow}\subsection{Renormalization} \label{sec:renor}
One can ask as the $L^2$-mass $M(t)$ goes to zero, whether the curve becomes circular. To study this problem we plan to show that the isoperimetric ratio $\frac{2M(t)}{L^2(t)}$ goes to the optimal constant in the Wirtinger's inequality $\frac{1}{4\pi^2}$ as $t\rightarrow t^\ast$. For that and for many other purposes it is convenient to renormalize the flow.

We first introduce a slow time variable. More precisely, for $t\in [0,t^\ast)$ let
\begin{align*}
\tau(t):=-\ln L(t).
\end{align*}
Note that by Proposition \ref{prop:length}(i) and Corollary \ref{cor:ext2}, $\tau(t)$ is monotone increasing in $t$ and $\tau\rightarrow +\infty$ iff $t\rightarrow t^\ast$; this is also clear in view of  \eqref{eq:gfr}. Next we consider the normalization
\begin{align*}
\xi(\tau, s ):=\frac{\eta(t(\tau),s)}{L(t(\tau))}
\end{align*}
One advantage of using such renormalization is that the curve $\xi(\tau,s)$ has the unit speed parametrization, i.e.,
\begin{align}
%V(\tau)&:=\int_0^1|\xi(\tau,s)|^2 ds= \frac{2M(t)}{L(t)^2}\\
%&\geq \frac{\int_0^1|\eta(0,s)|^2 ds}{L_0^2} \text{ for all } \tau\in [0,\infty),\\
L(\tau)&=|\p_s\xi(\tau,s)|=1 \text{ for all } (\tau,s)\in [0,\infty)\times \Sf.\label{eq:constrxi}
\end{align}
A direct computation shows that $\xi$ satisfies the equation
\begin{align}\label{eq:xi0}
\p_\tau\xi =\p_s(\sigma\p_s\xi)+\xi.
\end{align}
Here $\sigma(\tau, s)$ can be viewed as a Lagrange multiplier coming from the constraint \eqref{eq:constrxi}. It satisfies
\begin{align}\label{eq:sigma0}
\p_{ss}\sigma-\sigma|\p_{ss}\xi|^2=-1, \quad \sigma(\tau, 0)=\sigma(\tau,1).
\end{align}
Indeed, from the change of variable $\frac{d\tau}{dt}=-\frac{\p_tL(t)}{L(t)}$. Thus
\begin{align*}
\p_\tau\xi=\frac{\p_t\eta}{L}\frac{d t}{d \tau}-\eta\frac{\p_tL}{L^2}\frac{d t}{d \tau}=-\frac{\p_t\eta}{\p_tL}+\frac{\eta}{L}.
\end{align*}
Using the equation of $\eta$ we have
\begin{align*}
\p_\tau \xi =-\frac{\p_s(\tilde{\sigma}\p_s\eta)}{L^2\p_tL}+\frac{\eta}{L}.
\end{align*}
Letting $\sigma:=\frac{\tilde{\sigma}}{-L\p_tL}$ and writing the above equation in terms of $\xi$ we arrive at \eqref{eq:xi0}. To derive \eqref{eq:sigma0} we can either use the equation of $\tilde{\sigma}$, or use the above equation of $\xi$ together with the constraint $|\p_s\xi|\equiv 1$.

With the same argument as in Section \ref{sec:grad} it is not hard to see that the normalized flow \eqref{eq:xi0} and \eqref{eq:sigma0} can be viewed as the positive gradient flow of the $L^2$-mass $M(\xi):=\frac{1}{2}\int_0^1|\xi(s)|^2$ on the manifold of immersed curves with arc-length parametrization, cf. \eqref{volpres},
\begin{align*}
\tilde{\mathcal{A}}:=\{\xi\in H^2(\mathbb{S}^1;\R^d): \ |\p_s\xi(s)|=1 \text{ for all } s\in \Sf, \ \int_0^1\xi(s)ds =0\}
\end{align*}
with respect to the $L^2(\mathbb{S}^1;\R^d)$-induced metric (see also Appendix \ref{sec:mult}  and our recent work \cite{ShV17} which analyzes the gradient flow of the potential energy on a space very similar to $\tilde{\mathcal{A}}$).

The normalized flow \eqref{eq:xi0}--\eqref{eq:sigma0} can be interpreted in the spirit of \cite{ShV17} as an overdamped motion of an inextensible loop whose particles are repelled from the origin with the force equal to the radius vector.

Using either the gradient flow structure for the normalized flow or tracing the change of variables and normalization, one obtains the monotonic quantities along the flow.

\begin{prop}\label{prop:mono}
Let $\xi$ be a solution to the normalized flow \eqref{eq:xi0}--\eqref{eq:sigma0}. Then
\begin{itemize}
\item[(i)] $\tau\mapsto \int_0^1|\xi(\tau,s)|^2 ds$ is monotone increasing. Moreover, $\int_0^1|\xi(\tau,s)|^2ds \leq \frac{1}{4\pi^2}$ for all $\tau$.
\item [(ii)] $\tau\mapsto \int_0^1\sigma(\tau,s)ds$ is monotone increasing with $\int_0^1\sigma(\tau,s)ds\leq \int_0^1|\xi(\tau,s)|^2ds$ for all $\tau$.
\item [(iii)] As $\tau\rightarrow \infty$ we have $\int_0^1\sigma(\infty,s)ds=\int_0^1|\xi(\infty,s)|^2ds$.
\end{itemize}
\end{prop}

\begin{proof}
1. We multiply \eqref{eq:xi0} by $\p_\tau \xi$ and integrate in $s$ from $0$ to $1$. After an integration by part and using $\p_s\xi\cdot \p_{s\tau}\xi=0$ (which follows from $|\p_s\xi|=1$) we obtain that
\begin{align}\label{eq:identity1}
\int_0^1 |\p_\tau \xi|^2 ds = \p_\tau\left(\frac{1}{2}\int_0^1|\xi|^2 ds\right).
\end{align}
Thus $\tau\rightarrow \int_0^1|\xi|^2$ is monotone increasing. The upper bound follows from the Wirtinger's inequality
\begin{align*}
\int_{0}^1|\xi|^2 \ ds \leq \frac{1}{4\pi^2} \int_0^1|\p_s\xi|^2 ds = \frac{1}{4\pi^2}.
\end{align*}
Again we have used that $\int_0^1\xi(\tau,s)ds=0$.\\

2. From the definition of $\sigma$ we have
\begin{align*}
\int_0^1\sigma(\tau,s)ds=\frac{\int_0^1\tilde{\sigma}ds}{-L\p_tL}=\frac{1}{-L\p_tL(t(\tau))}.
\end{align*}
By (ii) of Proposition \ref{prop:length}, $-L\p_tL$ is monotone decreasing in $t$ thus in $\tau$. Thus $\int_0^1\sigma(\tau,s)ds$ is monotone increasing in $\tau$. To prove the upper bound, we multiply \eqref{eq:xi0} by $\xi$ and integrate in $s$. Integrating by parts and using that $|\p_s\xi|=1$ we have
\begin{align}\label{eq:identity2}
\frac{1}{2}\p_t\int_0^1|\xi|^2 ds =\int_0^1|\xi|^2 ds-\int_0^1\sigma ds.
\end{align}
By (i) the left hand side is larger than or equal to zero. Thus the upper bound follows.\\

3. By (i) and (ii),  $\lim_{\tau\rightarrow\infty}\int_0^1|\xi(\tau,s)|^2ds$
 and $\lim_{\tau\rightarrow \infty}\int_0^1\sigma(\tau,s)ds$ exist and in the limit $\int_0^1|\xi(\infty,s)|^2ds-\int_0^1\sigma(\infty,s)ds\geq 0$. To see the limit is actually zero, we argue by contradiction. Suppose not, then there exist $\epsilon>0$ and $M_\epsilon>0$ such that $\int_0^1|\xi(\tau,s)|^2ds-\int_0^1\sigma(\tau,s)ds\geq \epsilon$ for all $\tau\geq M_\epsilon$. By \eqref{eq:identity2} for any $\tau_1>\tau_2\geq M_\epsilon$
\begin{align*}
\int_0^1|\xi(\tau_1,s)|^2ds-\int_0^1|\xi(\tau_2,s)|^2ds\geq 2\epsilon(\tau_1-\tau_2).
\end{align*}
By (i) the left hand side is bounded from above by $\frac{1}{4\pi^2}$. However, the right hand side goes to infinity as $\tau_1\rightarrow\infty$, which is a contradiction. Thus we have shown that the limit is zero.
\end{proof}

\vspace{5mm}

In the rest of the paper we will work with the normalized flow \eqref{eq:xi0}-\eqref{eq:sigma0}. In Section \ref{sec:local}, we show local well-posedness of the problem in some H\"older class. In Section \ref{sec:stat}, we explore the stationary solutions.  In Section \ref{sec:global}, we address global well-posedness and long time asymptotics of the solution for the initial data which are close to steady states, and show exponential decay of the solutions to a steady state. In Section \ref{sec:weak} we address the global solvability without restrictions on the initial data but in a generalized sense. We stress that there is a one-to-one correspondence between the normalized flow and the original gradient flow (Sections \ref{sec:grad} and \ref{sec:normalize_flow}, resp.). After the backward change of variable and renormalization, we can infer the well-posedness of the original flow and the asymptotics of the flow near the extinction time.

\subsection{Local well-posedness}\label{sec:local}
In this section we show the local well-posedness of the normalized flow \eqref{eq:xi0}--\eqref{eq:sigma0}.
First we introduce the function spaces we will work with.
Given $\alpha,\beta\in [0,1)$, $T\in (0,\infty)$ and $k\in \N\cup\{0\}$, let
\begin{align*}
C^{k+\alpha, \beta}([0,T]\times \Sf):=\{\xi:[0,T]\times\Sf\rightarrow \R^d: \|\xi\|_{k+\alpha,\beta}<\infty\},
\end{align*}
where
\begin{align*}
\|\xi\|_{k+\alpha,\beta}:=\sup_{t}\|\xi(t,\cdot)\|_{C^{k+\alpha}(\Sf)}+\sup_s\sum_{j=0}^k\|\p_s^j\xi(\cdot,s)\|_{C^{\beta}([0,T])}.
\end{align*}
Here we use $C^{k+\alpha}(\Sf)$ ($C^{\beta}([0,T])$) to denote the usual H\"older spaces for functions only depending on one variable. Similarly, let
\begin{align*}
C^{k+\alpha,1+\beta}([0,T]\times \Sf):=\{\xi:[0,T]\times \Sf\rightarrow \R^d: \|\xi\|_{k+\alpha,\beta}+\|\p_t\xi\|_{k+\alpha,\beta}<\infty\}.
\end{align*}

The local well-posedness result we want to prove in this section is as follows.

\begin{thm}\label{thm:wellposedness}
Given any initial datum $\xi_0\in C^{2+\alpha}(\Sf)$ with $|\p_s\xi_0(s)|=1$, $\int_0^1\xi_0(s)ds=0$ , there exists $T>0$, which depends on $\|\xi_0\|_{C^{2+\alpha}(\Sf)}$, such that  the Cauchy problem
\begin{equation}\label{eq:wellposed}
\begin{split}
\p_t\xi=\p_s(\sigma\p_s\xi)+\xi,\\
\p_{ss}\sigma-|\p_{ss}\xi|^2\sigma=-1,\\
\xi(0,s)=\xi_0(s)
\end{split}
\end{equation}
has a unique solution $\xi\in C^{2+\alpha,\alpha/2}([0,T]\times \Sf)$, $\p_t\xi\in C^{\alpha, \alpha/2}([0,T]\times \Sf)$.
\end{thm}

\begin{rmk}
As in Remark \ref{rmk:constraint}, if $\xi$ is the solution emanating from $\xi_0$ provided by Theorem \ref{thm:wellposedness}, then $\int_0^1\xi(t,s)ds=0$ and $|\p_s\xi(t,s)|\equiv 1$ for all $t\in [0,T]$.
\end{rmk}

The proof of Theorem \ref{thm:wellposedness} is based on the Banach fix point theorem, where we show the solution map $\xi\mapsto \sigma_\xi\mapsto \tilde{\xi}$ is a contraction in the Banach space $C^{2+\alpha,\alpha/2}([0,T]\times\Sf)$. The proof is divided into several lemmas.

\begin{lem}\label{lem:sigma}
For any $\xi\in C^{2+\alpha, \alpha/2}([0,T]\times \Sf)$ such that $\|\p_{ss}\xi(t,\cdot)\|_{L^2(\Sf)}\neq 0$ for all $t\in [0,T]$, there exists a unique solution $\sigma\in C^{2+\alpha/2,\alpha/4}([0,T]\times \Sf)$ to the ODE
\begin{equation}\label{eq:eta_sigma}
\begin{split}
&\p_{ss}\sigma (t,s)-\sigma(t,s)|\p_{ss}\xi(t,s)|^2=-1.
\end{split}
\end{equation}
The solution satisfies the estimate
\begin{align*}
\|\sigma\|_{2+\alpha/2,\alpha/4}\leq C
\end{align*}
for some constant $C$ depending on $d$ and $\|\p_{ss}\xi\|_{\alpha,\alpha/2}$, which is uniformly bounded if $\|\p_{ss}\xi\|_{\alpha,\alpha/2}$ is uniformly bounded.
\end{lem}
\begin{proof}
Let $k(t,s):=|\p_{ss}\xi(t,s)|$. Since $\|k(t,\cdot)\|_{L^2}\neq 0$, the inhomogeneous equation has a unique solution $\sigma(t,s)$ for each $t$. Furthermore, by the regularity theory for the elliptic equations $\sigma(t,\cdot)\in C^{2+\alpha}(\Sf)$ for each $t\in [0,T]$ and one has the estimate
\begin{align}\label{eq:est000}
\|\sigma\|_{2+\alpha,0}\leq C\left(\|k\|_{\alpha,0}+1\right)
\end{align}
for some universal $C>0$ (cf. Lemma \ref{lem:lower_sigma} below for the estimate of $\|\sigma\|_{L^\infty}$).

To derive the regularity in $t$ we consider the equation for $\sigma(t_1,\cdot)-\sigma(t_2,\cdot)$ for any $0\leq t_1<t_2\leq T$:
\begin{align*}
&\quad\p_{ss}\left(\sigma(t_1,s)-\sigma(t_2,s)\right)-k(t_1,s)^2\left(\sigma(t_1,s)-\sigma(t_2,s)\right)\\
&=\sigma(t_2,s)\left(k^2(t_1,s)-k^2(t_2,s)\right).
\end{align*}
By the Schauder estimate,
\begin{align*}
\left\|\sigma(t_1,s)-\sigma(t_2,s)\right\|_{2+\alpha/2,0}\leq \tilde{C}\|k(t_1,s)-k(t_2,s)\|_{\alpha/2,0},
\end{align*}
where $\tilde{C}$ is a constant depending on $\|\sigma\|_{\alpha/2,0}$ and $\|k\|_{\alpha/2,0}$, and it is uniformly bounded if $\|k\|_{\alpha/2,0}$ is uniformly bounded (recall the estimate \eqref{eq:est000} for $\sigma$). Here we have used
 $$\|\sigma(t_2,s)\left(k^2(t_1,s)-k^2(t_2,s)\right)\|_{\alpha/2,0}\leq 6\|\sigma\|_{\alpha/2,0}\|k\|_{\alpha/2,0}\|k(t_1,\cdot)-k(t_2,\cdot)\|_{\alpha/2,0}$$ to estimate the right hand side.
Since $k(t,s)\in C^{\alpha,\alpha/2}([0,T]\times \Sf)$, it is not hard to see that
\begin{align*}
\|k(t_1,s)-k(t_2,s)\|_{\alpha/2,0}\leq C\|k\|_{\alpha,\alpha/2}|t_1-t_2|^{\alpha/2-\alpha/4}
\end{align*}
for some universal $C>0$.
Combining the last two inequalities we obtain the desired estimate of $\sigma(t,s)$.
\end{proof}

Next we state the pointwise upper and lower bound on $\sigma(t,\cdot)$ in terms of $\|k(t,\cdot)\|_{L^2(\Sf)}$. This is a direct consequence of the upper and lower bound of the Green's function of the Schr\"odinger operator $\p_{ss}-k^2$ with the periodic boundary conditions (cf. Proposition A.3 in \cite{Preston-2012}).

\begin{lem}\label{lem:lower_sigma}
Let $\xi$ and $\sigma$ be the same as in Lemma \ref{lem:sigma}. Then
\begin{align*}
\frac{e^{-\rho/2}}{\rho}\leq \sigma(t,s)\leq 1+\frac{1}{\rho},\text{ where }\rho=\rho(t)=\int_0^1|\p_{ss}\xi(t,s)|^2 ds
\end{align*}
for all $s\in \Sf$.
\end{lem}

Fix an initial datum $\xi_0(s)$ as in Theorem \ref{thm:wellposedness}. Firstly we note that $\int_0^1|\p_{ss}\xi_0(s)|^2 ds \geq 4\pi^2$. To see this let $Z:=\p_s\xi_0$. From the assumptions on $\xi_0$ and using periodicity we have $|Z(s)|\equiv 1$ and $\int_0^1Z=0$. Then the claimed inequality follows from the Wirtinger's inequality $\int_0^1|\p_sZ|^2 ds\geq 4\pi^2\int_0^1|Z|^2ds =4\pi^2$.  Next we let
$\delta_0:=\frac{1}{2}\|\p_{ss}\xi_0\|_{L^2(\Sf)}\geq \pi$ and let
\begin{align*}
M_{\xi_0}:=\{\xi\in C^{2+\alpha,\alpha/2}([0,T]\times\Sf):&\ \|\xi(t,s)\|_{2+\alpha,\alpha/2}\leq \|\xi_0\|_{2+\alpha}+\delta_0,\\
&\text{ and }\|\p_{ss}\xi(t,\cdot)-\p_{ss}\xi_0(\cdot)\|_{L^2(\Sf)}\leq \delta_0\}.
\end{align*}
It is not hard to see that $M_{\xi_0}$ is a closed convex subset in $C^{2+\alpha,\alpha/2}([0,T]\times \Sf)$. 

\begin{lem}\label{lem:tilde_eta}
Given any $\xi\in M_{\xi_0}$, let $\sigma=\sigma_\xi$ be as in Lemma \ref{lem:sigma}. Then there is $T_0=T_0(\|\xi_0\|_{2+\alpha})$ sufficiently small, such that for any $T\in (0,T_0]$ there exists a unique solution $\tilde{\xi}(t,s)\in M_{\xi_0}$ to the initial value problem
\begin{equation}\label{eq:sigma_eta}
\begin{split}
\p_t\tilde{\xi}=\p_s(\sigma\p_s\tilde{\xi})+\tilde{\xi} \text{ in } (0,T)\times \Sf,\quad \tilde{\xi}\big|_{t=0}=\xi_0.
\end{split}
\end{equation}
\end{lem}

\begin{proof}
Given any $\xi\in M_{\xi_0}$, by the triangle inequality $\delta_0\leq \|\p_{ss}\xi(t,\cdot)\|_{L^2(\Sf)}\leq 3\delta_0$ for any $t\in [0,T]$. Thus by Lemma \ref{lem:sigma} there exists a unique solution $\sigma=\sigma_\xi$ to \eqref{eq:eta_sigma} in the class $C^{2+\alpha/2, \alpha/4}([0,T]\times\Sf)$. Moreover, it satisfies $\|\sigma\|_{2+\alpha/2,\alpha/4}\leq C$, where $C$ depends on $\|\xi_0\|_{2+\alpha}$ since $\xi\in M_{\xi_0}$.
% and it is uniformly bounded if $\|\xi\|_{2+\alpha,\alpha/2}$ is uniformly bounded.
From Lemma \ref{lem:lower_sigma} we have that $\sigma\geq c_0>0$ for some $c_0$ depending on $\|\xi_0\|_{2+\alpha}$.  Thus the equation \eqref{eq:sigma_eta} is parabolic. By the classical well-posedness results for the parabolic equations, there is a unique solution $\tilde{\xi}\in C^{2+\alpha/2,1+\alpha/4}([0,T]\times\Sf)$  to the equation \eqref{eq:sigma_eta}, and $\|\tilde{\xi}\|_{2+\alpha/2,1+\alpha/4}\leq \tilde{C}$, where $\tilde{C}$ depends on $\|\xi_0\|_{2+\alpha}$ (cf. Section 9.2 in \cite{Kry96}). 

Next we claim there is a small enough $T_0>0$ depending on $\|\xi_0\|_{2+\alpha}$, such that $\tilde{\xi}\in M_{\xi_0}$ for any $T\in (0,T_0]$. Indeed, the definition of the H\"older class and an interpolation yield that
\begin{align*}
\|\tilde{\xi}\|_{2+\alpha,\alpha/2}\leq \|\xi_0\|_{2+\alpha}+ 2T^{1-\alpha/4}\|\tilde{\xi}\|_{2+\alpha/2,1+\alpha/4}\leq \|\xi_0\|_{2+\alpha}+ 2T^{1-\alpha/4}\tilde{C}.
\end{align*}
By the H\"older regularity of the solution,
$\|\p_{ss}\tilde{\xi}(t,\cdot)-\p_{ss}\xi_0(\cdot)\|_{L^2(\Sf)}\leq 2\tilde{C}t^{\alpha/2}$ for any $t\in [0,T]$. The claim then follows by taking $T_0$ sufficiently small depending on $\|\xi_0\|_{2+\alpha}$.
\end{proof}

\begin{lem}\label{lem:contraction}
Let $\xi, \tilde{\xi}\in M_{\xi_0}$ be as in Lemma \ref{lem:tilde_eta}. Then there is $T>0$ sufficiently small depending on $\|\xi_0\|_{2+\alpha}$, such that the mapping $\xi\mapsto \tilde{\xi}$ is a contraction.
\end{lem}
\begin{proof}
Given $\xi_1, \xi_2\in M_{\xi_0}$, let $\sigma_1$ and $\sigma_2$ be the solutions to \eqref{eq:eta_sigma} with respect to $\xi_1$ and $\xi_2$ correspondingly. Let $k_i:=|\p_{ss}\xi_i|$, $i=1,2$. Then $\sigma_1-\sigma_2$ satisfies the equation
\begin{align*}
\p_{ss}(\sigma_1-\sigma_2)-k_1^2(\sigma_1-\sigma_2)=\sigma_2(k_1^2-k_2^2).
\end{align*}
 By the similar arguments as in Lemma \ref{lem:sigma} (with slightly more involved estimates when dealing with the regularity in time due to the more complicated right hand side) we have
\begin{align}\label{eq:contract_1}
\|\sigma_1-\sigma_2\|_{2+\alpha/2,\alpha/4}\leq C\|\xi_1-\xi_2\|_{2+\alpha,\alpha/2}
\end{align}
for some $C=C(\|\xi_0\|_{2+\alpha})$.

Let $\tilde{\xi}_1, \tilde{\xi}_2$ be the solutions to \eqref{eq:sigma_eta} with respect to $\sigma_1, \sigma_2$ correspondingly. Then $\zeta:=\tilde{\xi}_1-\tilde{\xi}_2$ satisfies the equation
\begin{align*}
\p_t\zeta=\p_s(\sigma_1\p_s\zeta)+\zeta+\p_s\left((\sigma_1-\sigma_2)\p_s\tilde{\xi}_2\right),\quad \zeta(0,s)=0.
\end{align*}
By the parabolic Schauder estimate we have
\begin{align}\label{eq:contract_2}
\|\zeta\|_{3+\alpha/2,1+\alpha/4}\leq C\|\sigma_1-\sigma_2\|_{2+\alpha/2,\alpha/4}
\end{align}
for some $C=C(\|\xi_0\|_{2+\alpha})$. An interpolation together with \eqref{eq:contract_2} and \eqref{eq:contract_1}  yields
\begin{align*}
\|\zeta\|_{2+\alpha,\alpha/2}&\leq CT^{1-\alpha/4}\|\zeta\|_{3+\alpha/2,1+\alpha/4}\\
&\leq CT^{1-\alpha/4}\|\sigma_1-\sigma_2\|_{2+\alpha/2,\alpha/4}\\
&\leq CT^{1-\alpha/4}\|\xi_1-\xi_2\|_{2+\alpha,\alpha/2}.
\end{align*}
Here $C$ might be different from line to line but all depend on $\|\xi_0\|_{2+\alpha}$. Choosing $T$ to be sufficiently small depending on $\|\xi_0\|_{2+\alpha}$ we obtain
\begin{align*}
\|\tilde{\xi}_1-\tilde{\xi}_2\|_{2+\alpha,\alpha/2}=\|\zeta\|_{2+\alpha,\alpha/2}\leq \frac{1}{2}\|\xi_1-\xi_2\|_{2+\alpha,\alpha/2},
\end{align*}
which completes the proof.
\end{proof}

In the end we provide a proof for the local-posedness of our problem \eqref{eq:wellposed}.

\begin{proof}[Proof of Theorem \ref{thm:wellposedness}]
%The assumptions on $\xi_0$ yield that $\|\xi_0\|_{2+\alpha,\alpha/2}\neq 0$.
By Lemma \ref{lem:contraction} the mapping $\xi\mapsto \tilde{\xi}$ is a contraction on $M_{\xi_0}$ provided $T$ is sufficiently small depending on $\xi_0$. By the Schauder fixed point theorem, there is a unique function $\xi\in M_{\xi_0}$ with $\xi=\tilde{\xi}$. It is not hard to see that such fixed point $\xi$ is a solution to the initial value problem \eqref{eq:wellposed}. The solution is H\"older $\alpha/2$ continuous in $t$ up to $t=0$. The regularity of $\p_t\xi$ for $t>0$ follows immediately from the interior regularity of the parabolic equation.
\end{proof}

\subsection{Stationary solutions to the normalized flow} \label{sec:stat}
Our goal is to study the global well-posedness of the normalized equation \eqref{eq:xi0}--\eqref{eq:sigma0}, and the long time asymptotics of the solution. Before that we investigate the stationary solutions to the normalized equation, and show in this case the Lagrange multiplier $\sigma$ satisfies an ODE which has a first integral. Just like as for the conventional curve-shortening flow \cite{AL86}, the stationary solutions are not necessarily circles in our case. However we will show that the circle (with $\sigma \equiv \frac{1}{4\pi^2}$) is the only solution to the ODE if one assumes that the curve is simple and $\int_0^1\sigma ds$ satisfies a lower bound $\int_0^1\sigma ds \geq \frac{27}{32}\frac{1}{4\pi^2}$.\\

We start by recalling the stationary equation $\xi:\mathcal S^1\rightarrow \R^d$
\begin{equation}\label{eq:stationary}
\begin{split}
\p_s(\sigma\p_s\xi)+ \xi&=0,\ |\p_s\xi|=1.
\end{split}
\end{equation}

\begin{prop}\label{prop:stat}
Let $\xi\in C^2(\Sf;\R^d)$ be a solution to \eqref{eq:stationary}. Then $\xi$ is a plane curve whose curvature satisfies $k(s)>0$ for all $s$.
\end{prop}
\begin{proof}
Differentiating the equation in $s$ and using the constraint $|\p_s\xi|=1$ it is not hard to see that $\sigma$ satisfies
\begin{align*}
\p_{ss}\sigma-\sigma |\p_{ss}\xi|^2 = -1.
\end{align*}
We have $\sigma\in C^2(\Sf)$ by the elliptic estimate. Furthermore, by the strong maximum principle $\sigma>0$. This together with the equality $-\xi=\p_s\sigma\p_s\xi+\sigma\p_{ss}\xi$ implies that $\xi$ is a plane curve, since $\xi, \p_s\xi $ and $\p_{ss}\xi$ are in the same plane.

For the arc-length parametrized curve we have $\p_{ss}\xi=k N$, where $N$ is the unit inner normal along the curve and $k$ is the curvature.  Differentiating the equation of $\xi$ and using $\p_sN=-k\p_s\xi$ we get
\begin{align*}
\left(\p_{ss}\sigma-\sigma k^2+1\right)\p_s\xi + \left(2\p_s\sigma k +\sigma \p_sk\right)N=0.
\end{align*}
Thus $\sigma \p_s k+2\p_s \sigma  k =0$, which implies
\be
\sigma^2 k=const.
\ee
Since $\int_0^1k(s)ds=2\pi$  for a regular closed plane curve, and $\sigma>0$, one has $k(s)>0$ for all $s\in \Sf$.
\end{proof}

In the next proposition we derive a first-order ODE of $\sigma$.
\begin{prop}\label{prop:ode_sigma}
Let $\xi \in C^{2}(\Sf;\R^d)$ be a solution to \eqref{eq:stationary}. Let $\tau:=\sigma^2$. Then $\tau$ satisfies the first integral
\begin{align*}
\frac{1}{2}(\p_s\tau)^2+V(\tau)=\lambda,\quad V(\tau):=4\tau^{3/2}-6\bar \tau \tau.
\end{align*}
Here $\bar\tau=\int_0^1\tau^{1/2}ds$ and  $\lambda=V(\tau_e)$ is a fixed constant, where $\tau_e$ is any extreme value of $\tau$. Moreover, we have $\lambda\in [-2\bar\tau^3,0)$.
\end{prop}
\begin{proof}
First multiplying \eqref{eq:stationary} by $\xi$ and an integration yield
\begin{align}\label{eq:sig_xi}
\int_0^1\sigma ds =\int_0^1|\xi|^2 ds.
\end{align}
Then we multiply \eqref{eq:stationary} by $\p_s\xi$ and use $|\p_s\xi|=1$ to obtain
\begin{align*}
\p_s\sigma+\frac{1}{2} \p_s|\xi|^2=0 \text{ for all } s\in \Sf,
\end{align*}
which together with \eqref{eq:sig_xi} gives
\begin{align}\label{eq:sig_0}
\sigma+\frac{1}{2} |\xi|^2=\frac{3}{2}\bar\sigma,\quad \bar\sigma:=\int_0^1\sigma.
\end{align}

On the other hand, \eqref{eq:stationary} together with the orthogonality $\p_s\xi\cdot \p_{ss}\xi=0$ yields
\begin{align}\label{eq:sig_1}
|\p_s\sigma|^2+\sigma^2k^2 =|\xi|^2.
\end{align}
Since $\sigma k^2=\p_{ss}\sigma+1$ by the equation of $\sigma$, we obtain from \eqref{eq:sig_1}
\begin{align*}
|\p_s\sigma|^2+\sigma(\p_{ss}\sigma + 1)=\frac{1}{2}\p_{ss}(\sigma^2) + \sigma=|\xi|^2,
\end{align*}
which together with \eqref{eq:sig_0} gives the ODE of $\sigma$
\begin{align}\label{eq:sig_2}
\frac{1}{2}\p_{ss}(\sigma^2)=3\bar\sigma- 3 \sigma.
\end{align}

Remember that we have set $\tau:=\sigma^2$. Let us rewrite the above equation in terms of $\tau$ as $\p_{ss}\tau =6\bar\tau -6\tau^{1/2}$, where $\bar\tau=\bar \sigma=\int_0^1\tau^{1/2} ds$. Multiplying both sides by $\p_s\tau$ and integrating we obtain
\begin{align}\label{eq:tau}
\frac{1}{2}(\p_s\tau)^2+V(\tau)=\lambda,\quad V(\tau):=4\tau^{3/2}-6\bar \tau \tau.
\end{align}
At $s$ such that $\tau(s)=\tau_e$ we have $\p_s\tau(s)=0$, thus $\lambda=V(\tau_e)$ from \eqref{eq:tau}. The potential $V$ satisfies $V''(\tau)=3\tau^{-1/2}$, hence it is convex on $(0,\infty)$. We note from \eqref{eq:sig_0} and the definition $\tau=\sigma^2$ that $\tau\in (0,3\bar\tau/2)$. This implies $V(\tau)\in [-2\bar \tau^3, 0)$ with $\min V(\tau)=V(\bar\tau^2)=-2\bar\tau^3$. Thus $\lambda=V(\tau_e)\in [-2\bar \tau^3, 0)$.
\end{proof}

If $\xi$ is an $m$-covered circle, $m\in \N$, then it is easy to see from \eqref{eq:sig_0}--\eqref{eq:sig_1} that $\sigma=\frac{1}{4\pi^2m^2}$. In general it is possible to apply the method for the proof of Theorem A in \cite{AL86} to classify solutions $\tau$ (hence $\sigma$) to the ODE \eqref{eq:tau}. In the next proposition we show that if $\xi$ is simple and $\sigma$ is close to $\frac{1}{4\pi^2}$ ($m=1$), then $\xi$ is a circle and $\sigma=\frac{1}{4\pi^2}$.

\begin{prop}\label{prop:limit_sol}
Let $\xi \in C^{2}(\Sf;\R^d)$ be a solution to \eqref{eq:stationary}. Assume $\xi$ is simple, i.e., $\xi(s_1)\neq \xi(s_2)$ for $s_1\neq s_2$. Assume $\int_0^1\sigma\ ds\geq \frac{27}{32}\frac{1}{4\pi^2}$. Then $\sigma \equiv \frac{1}{4 \pi^2}$, and $\xi$ is a circle centered at $0$ with radius $\frac{1}{2\pi}$.
\end{prop}
\begin{proof}
1. By Proposition \ref{prop:ode_sigma} $\xi$ is a plane curve with curvature $k>0$. Moreover, $\xi$ is simple by assumption. Thus by the four-vertex theorem $k(s)$ has at least four critical points. Let $0\leq s_1<s_2<\cdots < s_J\leq 1$ with index $J\geq 4$ be the critical points. Since $\sigma^2 k = const$ and $\sigma>0$, $\sigma$ has the same critical points at $s_i$, i.e., $\p_s\sigma(s_i)=0$ for $i\in J$. We claim that
\begin{align}\label{eq:ave}
\frac{1}{|I_i|}\int_{I_i}\sigma(s)ds =\bar\sigma,\quad I_i=(s_i,s_{i+1}),\ |I_i|=s_{i+1}-s_i, \ i\in J.
\end{align}
Indeed, an integration of \eqref{eq:sig_1} in $s$ over $I_i$ together with \eqref{eq:sig_0} gives
\begin{align*}
\int_{I_i}(\p_s\sigma)^2  + \int_{I_i}\sigma^2k^2  =\int_{I_i} |\xi|^2  =  \int_{I_i}(3\bar\sigma-2\sigma).
\end{align*}
On the other hand, multiplying the equation of $\sigma$ by $\sigma$ and an integration by parts yield,
\begin{align*}
\int_{I_i}(\p_s\sigma)^2 + \int_{I_i}\sigma^2k^2= \int_{I_i}\sigma.
\end{align*}
Here we have used $\p_s\sigma(s_i)=0$, hence the boundary term in the integration by parts vanishes. From the above two equalities we conclude
\begin{align*}
\int_{I_i}(3\bar\sigma-2\sigma)= \int_{I_i}\sigma.
\end{align*}
Thus \eqref{eq:ave} follows.\\

2. We show that if $\bar\sigma \geq \frac{27}{32}\frac{1}{4\pi^2}$, then $\sigma(s)\equiv\bar\sigma$ in the intervals $I_i$ with $|I_i|\leq \frac{1}{4}$. In particular, since $J\geq 4$ there is always an open interval where $\sigma\equiv \bar\sigma$ there.

Take $I_i$ such that $|I_i|\leq \frac{1}{4}$. For simplicity we write $I$ instead of $I_i$ in the sequel. We multiply both sides of \eqref{eq:sig_2} by $\sigma$ and integrate from $s_i$ to $s_{i+1}$. An integration by parts gives
\begin{align}\label{eq:int}
\int_{I}\sigma(\p_s\sigma)^2 ds =  3\int_{I} (\sigma^2-\bar\sigma\sigma) ds.
\end{align}
On the other side, by a generalized Beckner type inequality (see \cite[Lemma 4]{CJS16} with $q=\frac{4}{3}$, $p=\frac{3}{2}$ and $f=\sigma^{3/2}$; see also \cite{KMV16A} for a link with ``unbalanced optimal transport'') we have
\begin{align}\label{eq:Beckner}
|I|^{-1/2}\|\sigma\|_{L^2(I)}\left(\|\sigma\|_{L^2(I)}^2 -|I|^{-1} \|\sigma\|_{L^1(I)}^2\right)\leq \frac{9}{2}\frac{|I|^2}{4\pi^2}\int_I \sigma (\p_s\sigma)^2 ds.
\end{align}
Here we have used that $C_P=\frac{|I|^2}{4\pi^2}$ is the optimal Poincar\' e constant with respect to the interval $I$. Combining \eqref{eq:int} and \eqref{eq:Beckner}, we arrive at
\begin{align*}
|I|^{-1/2}\|\sigma\|_{L^2(I)}\left(\|\sigma\|_{L^2(I)}^2 -|I|^{-1} \|\sigma\|_{L^1(I)}^2\right)\leq \frac{27}{2}\frac{|I|^2}{4\pi^2} \left(\|\sigma\|_{L^2(I)}^2 - |I|^{-1} \|\sigma\|_{L^1(I)}^2\right).
\end{align*}
Note that if $|I|\leq \frac{1}{4}$ and $\bar\sigma> \frac{27}{32}\frac{1}{4\pi^2}$, then
\begin{align*}
|I|^{-1/2}\|\sigma\|_{L^2(I)}> \frac{27}{2}\frac{|I|^2}{4\pi^2}.
\end{align*}
Indeed, this immediately follows from H\"older's inequality and \eqref{eq:ave}:
 \begin{align*}
 |I|^{-1/2}\|\sigma\|_{L^2(I)}\geq |I|^{-1}\|\sigma\|_{L^1}=\bar\sigma>\frac{27}{2}\frac{|I|^2}{4\pi^2}.
 \end{align*}
Thus $\|\sigma\|_{L^2(I)}^2=|I|^{-1}\|\sigma\|_{L^1(I)}^2$. From the equality case of the H\"older's inequality $|\sigma|=const$ a.e. in $I$. This together with the continuity of $\sigma$ yields that $\sigma(s)\equiv \bar\sigma$ in $I$. \\

3. In the last step we show that $\sigma(s)=1/(4\pi^2)$ for all $s\in \Sf$. Indeed, from step 2 above there exists an interval, say, $(0,s_0)$ for some small $s_0>0$ such that $\sigma\equiv \bar \sigma$ there. Let $\tau:=\sigma^2$. By the Picard theorem the initial value problem ($\tau=\sigma^2$)
\begin{align*}
\p_{ss}\tau = 6 (\bar\sigma-\sqrt{\tau}),\quad \tau(s_0)=\bar\sigma^2, \quad \p_s\tau(s_0)=0
\end{align*}
has a unique solution in $(s_0,s_0+\delta)$ for some $\delta>0$.
Since the constant function $\tau\equiv \bar\sigma^2$ is a solution, thus necessarily $\tau=\bar\sigma^2$ in $(s_0, s_0+\delta)$. This shows that $\tau$, thus $\sigma$, is identically $\bar\sigma$ in the whole circle $\Sf$.

With this at hand \eqref{eq:sig_0} yields that $|\xi|^2=const=\bar\sigma$ in $\Sf$. Thus $\xi$ is a circle centered at the origin. Since the length of the curve is equal to $1$, then $|\xi|=\frac{1}{2\pi k}$ for $k\in \{1,2,\cdots \}$. By our assumption $\bar \sigma \geq \frac{27}{32}\frac{1}{4\pi^2}$ (or because the curve is simple), we should have $|\xi|=\frac{1}{2\pi}$ and $\bar\sigma=\frac{1}{4\pi^2}$.
\end{proof}

\begin{rmk}\label{rmk:simple}
From the proof of Proposition \ref{prop:limit_sol} one can see that we only need the existence of two critical points $s_1$, $s_2$ of the curvature function $k(s)$ with $|s_1-s_2|\leq 1/4$. This can also be achieved by assuming the symmetry property $\xi(s)=-\xi(s+1/2)$ instead of assuming that the curve is simple.
\end{rmk}

\subsection{Global well-posedness and exponential stability}\label{sec:global}
In this section we study the global well-posedness of the normalized flow \eqref{eq:xi0}--\eqref{eq:sigma0} under the assumption that initially the curve is $C^2$ close to the circle
\begin{equation}\label{eq:w0}
w_0(s):=\frac{1}{2\pi}(\cos (2\pi s), \sin (2\pi s), 0,\cdots, 0).
\end{equation}
We will mainly show the uniform (in time) boundedness of curvature $|\p_{ss}\xi|^2$, which yields that the time $T$ in the local well-posedness result has a uniform lower bound. The main idea of the proof is to show that under some smallness assumption at the initial time, the parabolicity is preserved along the flow, i.e., $\sigma>c_0$ pointwise for any $t<T$, where $c_0>0$ is some absolute constant.

The proof of the result is based on a dynamical system approach. We let
\begin{align*}
\mathcal{C}:=\{c w_0(s+\theta): c,\theta\in \R\}
\end{align*}
denote the manifold generated by $w_0$ which is invariant under the dilation and rotation.
For each $t\in (0,T)$ we decompose
\begin{align}\label{eq:tilde_xi}
\xi(t,s)=\tilde{\xi}(t,s)+c(t)w_0(s+\theta(t)),
\end{align}
where $c(t)w_0(\cdot+\theta(t))$ is the $L^2$ projection of $\xi(t)$ onto $\mathcal{C}$, i.e., for each $t$ fixed
\begin{align*}
c(t)w_0(\cdot+\theta(t))\in \argmin _{w\in \mathcal{C}} \left\{\int_{\Sf}|\xi(t,\cdot)-w(\cdot)| ^2\dH^1\right\}.
\end{align*}
We remark that minimum are achieved by considering the minimization problem over the finite dimensional parameter space
$$\inf_{c,\theta\in \R}\mathcal{F}_{\xi}(c,\theta),\quad \mathcal{F}_\xi(c,\theta):= \int_{\mathcal S^1}|\xi(\cdot)-cw_0(\cdot+\theta)|^2 \dH^1.$$
The first derivatives $\p\mathcal{F}_\xi(c,\theta)/\p c$ and $\p\mathcal{F}_{\xi}(c,\theta)/\p \theta$ vanish at the minimizers $(c(t),\theta(t))$, yielding the following orthogonality conditions
\begin{align}\label{eq:tilde_orth}
\int_0^1 \tilde{\xi}(t,s) \cdot w_0 (s+\theta(t)) ds = \int_0^1 \tilde{\xi}(t,s)\cdot c(t)\p_sw_0(s+\theta(t)) ds =0.
\end{align}
Note that since $c\mapsto \mathcal{F}_{\xi}(c,\theta)$ is strictly convex, there is indeed a unique $c(t)$ associated with $\xi(t,\cdot)$ for each $t\in [0,T]$.

Now we derive the evolution of the parameters $c(t)$ and $\theta(t)$ (assume for now they are differentiable. For more detailed discussion we refer to Lemma \ref{lem:differentiability}).  By using the equation of $\xi$ we obtain the equation of $\tilde{\xi}$
\begin{equation}\label{eq:tilde_equ}
\begin{split}
\p_t\tilde{\xi}(t,s)&=\tilde{\xi}(t,s)+\p_s(\sigma(t,s)\p_s\tilde{\xi}(t,s))+\\
&+c(t)\p_s\left((\sigma(t,s)-\frac{1}{4\pi^2})\p_s w_0(s+\theta(t))\right)-\dot{c}(t)w_0(s+\theta(t))-c(t)\dot{\theta}(t)\p_sw_0(s+\theta(t)),
\end{split}
\end{equation}
where we have used the relation $w_0+\frac{1}{4\pi^2}\p_{ss}w_0=0$.
We multiply \eqref{eq:tilde_equ} by $\tilde{\xi}(t,\cdot)$, $w_0(\cdot+\theta(t))$ and $\p_sw_0(\cdot+\theta(t))$ and integrate over $\Sf$. Using the orthogonality condition \eqref{eq:tilde_orth} we quantify the evolution of $\|\xi(t,\cdot)\|_{L^2}^2$ as well as of the parameters $c(t)$ and $\theta(t)$ for $t<T$:
\begin{equation}\label{eq:equ_eta}
\begin{split}
\frac{1}{2}\p_t\int_0^1|\tilde{\xi}|^2 ds &= \int_0^1|\tilde{\xi}|^2 -\int_0^1\sigma |\p_s\tilde{\xi}|^2 ds -c(t)\int_0^1\sigma\p_sw_0(\cdot+\theta(t))\cdot \p_s\tilde{\xi} ds, \\
\dot{c}(t)&=\left(1-4\pi^2\int_0^1\sigma ds\right)c(t)-4\pi^2\int_0^1\sigma \p_sw_0(\cdot+\theta(t))\cdot \p_s\tilde{\xi} ds,\\
c(t)\dot{\theta}(t)&=4\pi^2\int_0^1\sigma w_0(\cdot+\theta(t))\cdot \p_s\tilde{\xi} ds.
\end{split}
\end{equation}
From the unit speed constraint $|\p_s\xi(t,s)|=|\p_sw_0(s+\theta(t))|= 1$ for all $(t,s)\in [0,T]\times \Sf$, we obtain
\begin{equation}\label{eq:sqr}
|\p_s\tilde{\xi}(t,s)|^2+2c(t)\p_s\tilde{\xi}(t,s)\cdot \p_sw_0(s+\theta(t))+c(t)^2=1.
\end{equation}
 Integrating over $\Sf$ and using \eqref{eq:tilde_orth} (note that $\int_0^1\p_s\tilde{\xi}\cdot \p_sw_0 ds=-\int_0^1\tilde{\xi}\cdot \p_{ss}w_0=0$ since $\p_{ss}w_0=-4\pi^2w_0$) yields
\begin{equation}\label{eq:tilde_equ1}
\int_0^1|\p_s\tilde{\xi}|^2 ds =1-c(t)^2.
\end{equation}
Here and in the sequel for brevity we write $w_0$ and $\p_sw_0$ instead of $w_0(s+\theta(t))$ and $\p_sw_0(s+\theta(t))$, respectively. \\

\begin{lem}\label{lem:differentiability}
Suppose $\xi(t,s)\in C^{2+\alpha,\alpha/2}([0,T]\times \Sf)$, $\p_t\xi(t,s)\in C^{\alpha,\alpha/2}((0,T]\times \Sf)$ is a classical solution to the normalized flow \eqref{eq:xi0}--\eqref{eq:sigma0}. Suppose the $L^2$ projection of $\xi(0,\cdot)$ onto $\mathcal{C}$ is not $0$, i.e. $c(0)\neq 0$. Then there exist $t_0\in (0,T]$ depending on $c(0)$, $\|\xi(0,\cdot)\|_{C^2(\Sf)}$, and parameters $(c(t),\theta(t))\in C^0([0,t_0])\cap C^1((0,t_0])$, such that \eqref{eq:tilde_xi}--\eqref{eq:equ_eta} are satisfied for all $t\in (0,t_0)$.
\end{lem}

\begin{proof}
Given a classical solution $\xi$, there exist $(c(0),\theta(0))\in \R\times [0,2\pi)$ be such that $$\mathcal{F}_{\xi(0)}(c(0),\theta(0))=\min_{(c,\theta)\in \R\times [0,2\pi)} \mathcal{F}_{\xi(0)}(c,\theta).$$ Indeed, since $\mathcal{F}_{\xi}$ is continuous, strictly convex in $c$ and $2\pi$-periodic in $\theta$, the minimum is realized. We consider the ODE system
\begin{equation}\label{eq:ode_ctheta}
\begin{split}
\dot c(t) &= c(t)-4\pi^2 \int_0^1 \sigma \p_s w_0(s+\theta(t)) \cdot \p_s\xi(t,s) ds,\\
c(t)\dot\theta(t)& =4\pi^2 \int_0^1\sigma w_0(s+\theta(t))\cdot \p_s\xi(t,s)ds,\\
(c(t),\theta(t))\big |_{t=0} &=(c(0),\theta(0)).
\end{split}
\end{equation}
Using $|\p_s\xi|=1$ and the bound of $\sigma$ in Lemma \ref{lem:lower_sigma}, we have that the right-hand side of the system is bounded by constants depending on $\|\xi(0,\cdot)\|_{C^2(\Sf)}$. Thus by the Picard-Lindel\"of theorem there exists a unique solution $(c(t),\theta(t))$ defined in $[0,t_0]$ for some $t_0\in (0,T]$, whose value lies in a neighborhood of $(c(0),\theta(0))$. Here $t_0$ and the size of the neighborhood depend on $c(0)$ and $\|\xi(0,\cdot)\|_{C^2(\Sf)}$. Furthermore, from the $t$-regularity of $\xi$ and $\sigma$ we can conclude that $(c(t),\theta(t))\in C^0([0,t_0])\cap C^1((0,t_0])$.

Let $\tilde{\xi}(t,s):=\xi(t,s)-c(t)w_0(s+\theta(t))$, $t\in [0,t_0]$. From the equation of $\xi$ one easily derives the equation of $\tilde{\xi}$ in \eqref{eq:tilde_equ}. Moreover, using $|\p_sw_0|=1$ and $\p_sw_0\cdot w_0=0$, the ODE system \eqref{eq:ode_ctheta} can be rewritten in terms of $\tilde{\xi}$ as in \eqref{eq:equ_eta}. We claim that \eqref{eq:tilde_equ} together with the ODE for $(c(t),\theta(t))$ in \eqref{eq:equ_eta} implies the orthogonality conditions \eqref{eq:tilde_orth}. Indeed, let
$$A(t):=\int_0^1 \tilde{\xi}(t,s)\cdot w_0(s+\theta(t)) ds,\ B(t):=\int_0^1 \tilde{\xi}(t,s)\cdot \p_s w_0(s+\theta(t)) ds,\ t\in [0,t_0].
$$ We have $A(0)=B(0)=0$. Direct differentiation yields
\begin{equation}\label{eq:ode_ab}
\begin{split}
\dot A(t) &= \int_0^1\p_t\tilde{\xi}(t,s)\cdot w_0(s+\theta(t)) ds+ \dot\theta(t) B(t),\\
\dot B(t) &= \int_0^1 \p_t\tilde{\xi}(t,s) \cdot \p_s w_0(s+\theta(t))ds - \frac{\dot\theta(t)}{4\pi^2} A(t).
\end{split}
\end{equation}
 Multiplying \eqref{eq:tilde_equ} by $w_0(s+\theta(t))$ and $\p_s w_0(s+\theta(t))$, and using the ODE of $(c(t),\theta(t))$ in \eqref{eq:equ_eta}, we obtain
\begin{equation}\label{eq:ode_ab1}
\begin{split}
\int_{0}^1 \p_t\tilde{\xi}(t,s)\cdot w_0(s+\theta(t)) ds  = A(t),\ \int_0^1 \p_t\tilde{\xi}(t,s) \cdot \p_s w_0(s+\theta(t))ds = B(t).
\end{split}
\end{equation}
Applying \eqref{eq:ode_ab1} to \eqref{eq:ode_ab} yields
\begin{equation*}
\left(\begin{array}{ll}
\dot A(t)\\
\dot B(t)
\end{array}\right)=
\left(\begin{array}{ll}
1 & \dot\theta(t)\\
-\frac{\dot\theta(t)}{4\pi^2} & 1
\end{array}
\right)
\left(\begin{array}{ll}
 A(t)\\
 B(t)
\end{array}\right)
\end{equation*}
Since the coefficient matrix is nonsingular and since $A(0)=B(0)=0$ we have that $A(t)=B(t)=0$ for all $t\in [0,t_0]$, which completes the proof for the claim.
\end{proof}

%\begin{rmk}[Differentibility of the parameters $c(t)$ and $\theta(t)$]
%\label{rmk:parameter}
%Suppose $\xi(t,s)\in C^{2+\alpha,\alpha/2}([0,T]\times \Sf)$, $\p_t\xi(t,s)\in C^{\alpha,\alpha/2}((0,T]\times \Sf)$ is a classical solution and suppose $c(0)\neq 0$, then by the Picard-Lindel\"of theorem there exists a unique solution $(c(t),\theta(t))$ to the ODE \eqref{eq:equ_eta} with the initial value $(c(0),\theta(0))$ on the interval $[0,t_0)$ for some $t_0>0$. Furthermore, it is not hard to check that the orthogonality conditions \eqref{eq:tilde_orth} are preserved along the flow. The condition $c(0)\neq 0$ holds if, for example, the classical solution $\xi$ satisfies $\|\xi(0,\cdot)-w_0(\cdot)\|_{L^2(\Sf)}\leq\frac{1}{4}\frac{1}{4\pi^2}$. We will show later that under further lower bound assumptions on $\int_{\Sf}\sigma$ and $c$ at $t=0$ (cf. \eqref{eq:lower_bound_sigma} and \eqref{eq:lower_bound_c}), the parameter $c(t)$ remains uniformly bounded away from zero for all $t\in (0,t_0)$ (cf. Lemma \ref{lem:exp_decay}--Corollary \ref{cor:c}). This together with the monotonicity of $\int_{\Sf}\sigma(t,\cdot)$ (cf. Proposition \ref{prop:mono}) implies that the solution to the ODE exists up to time $T$.
%Thus in the sequel we can and will directly assume that $t_0=T$.
%\end{rmk}

The rest of the argument goes as follows:
\begin{itemize}
\item [(i)] in Lemma \ref{lem:exp_decay} and Lemma \ref{lem:exp_deri} we prove the decay estimate for $t\mapsto \|\xi(t,\cdot)\|_{L^2}$ and $t\mapsto \|\p_s\xi(t,\cdot)\|_{L^2}$, $t\in (0,t_0]$, under the initial assumptions that $\int_0^1\sigma(0,s)ds \geq \frac{2}{3}\frac{1}{4\pi^2}$ and $c(0)^2\geq \frac{1}{2}$. As a corollary of the decay estimate and Proposition \ref{prop:mono} (ii), these initial assumptions are preserved along the flow. We remark that the initial assumptions are always satisfied if initially $\xi$ is sufficiently close to the stationary solution $w_0$ (cf. Theorem \ref{thm:global}).
\item [(ii)] in Lemma \ref{lem:para} we derive pointwise oscillation estimate for the Lagrange multiplier $\sigma$. This is a crucial step, since $\sigma$ appears as the parabolicity coefficient in the equation of $\xi$.
\item [(iii)] Lemma \ref{lem:para} together with the decay estimate in (i) would yield the uniform (in $t_0$) boundedness of $\|\xi(t,\cdot)\|_{C^{2+\alpha}(\Sf)}$. Thus we have a lower bound on the time step in the iteration procedure. These are proved in Theorem \ref{thm:global}, where we show the global well-posedness results for initial datum sufficiently close to $w_0$.
\end{itemize}

In the sequel $t_0$, $\tilde{\xi}$ and $(c(t), \theta(t))$ are those given by Lemma \ref{lem:differentiability} associated with a classical solution $\xi(t,s)\in C^{2+\alpha, \alpha/2}([0,T]\times \Sf)$, $\p_t\xi(t,s)\in C^{\alpha,\alpha/2}((0,T]\times \Sf)$ with $c(0)\neq 0$. We also let
\begin{equation*}
\bar\sigma(t):=\int_0^1\sigma(t,s)ds.
\end{equation*}

\begin{lem}\label{lem:exp_decay}
Assume that
\begin{align}\label{eq:lower_bound_sigma}
\bar\sigma(0) \geq \frac{2}{3}\frac{1}{4\pi^2},
\end{align}
Then $t\mapsto \|\tilde{\xi}(t,\cdot)\|_{L^2(\Sf)}$ is monotone decreasing in $[0,t_0]$. Moreover,
\begin{align}\label{eq:diff}
\|\tilde{\xi}(t,\cdot)\|_{L^2(\Sf)}\leq e^{-\frac{1}{3}t}\|\tilde{\xi}(0,\cdot)\|_{L^2(\Sf)}, \quad t\in [0,t_0].
\end{align}
\end{lem}

\begin{proof}
We will derive a differential inequality on $\|\tilde{\xi}(t,\cdot)\|_{L^2(\Sf)}$ using its evolution equation from \eqref{eq:equ_eta}.
First we multiply $\sigma$ to the both sides of  \eqref{eq:sqr} and integrate over $\Sf$
\begin{align}\label{eq:sqr2}
 (1-c^2)\bar\sigma=\int_0^1\sigma|\p_s\tilde{\xi}|^2 ds +2c\int_0^1\sigma \p_s\tilde{\xi}\cdot \p_sw_0  ds.
\end{align}
Thus \eqref{eq:equ_eta} can be rewritten as
\begin{align*}
\frac{1}{2}\p_t\int_0^1|\tilde{\xi}|^2& = \int_0^1 |\tilde{\xi}|^2 - \int_0^1\sigma|\p_s\tilde{\xi}|^2 -\frac{1-c^2}{2}\bar\sigma +\frac{1}{2}\int_0^1\sigma|\p_s\tilde{\xi}|^2 ds\\
&=\int_0^1 |\tilde{\xi}|^2 - \frac{1}{2}\int_0^1\sigma|\p_s\tilde{\xi}|^2-\frac{1-c^2}{2}\bar\sigma.
\end{align*}
Applying \eqref{eq:tilde_equ1} to the last term of the above equation we obtain
\begin{equation}\label{eq:diff_inequ0}
\begin{split}
\frac{1}{2}\p_t\int_0^1|\tilde{\xi}|^2&=\int_0^1 |\tilde{\xi}|^2 - \frac{1}{2}\int_0^1\sigma|\p_s\tilde{\xi}|^2-\frac{1}{2}\bar\sigma\int_0^1|\p_s\tilde{\xi}|^2 ds\\
&\leq \int_0^1 |\tilde{\xi}|^2-\frac{1}{2}\bar\sigma\int_0^1|\p_s\tilde{\xi}|^2 ds.
\end{split}
\end{equation}
We claim that \eqref{eq:tilde_xi} implies
\begin{align}\label{eq:poincare}
\int_0^1|\tilde{\xi}(t,\cdot)|^2 ds\leq \frac{1}{16\pi^2}\int_0^1|\p_s\tilde{\xi}(t,\cdot)|^2 ds.
\end{align}
To see this, one can employ Fourier expansion of $\tilde{\xi}(t,\cdot)$. By $\int_0^1\tilde{\xi}=0$ as well as the orthogonality condition \eqref{eq:tilde_orth}, one finds that the zero and first order Fourier coefficients of $\tilde{\xi}$ are zero.

Applying \eqref{eq:poincare} to \eqref{eq:diff_inequ0} we obtain
\begin{align*}
\frac{1}{2}\p_t\int_0^1|\tilde{\xi}|^2\leq -\left(8\pi^2\bar\sigma(t)-1\right)\int_0^1 |\tilde{\xi}|^2.
\end{align*}
Since $t\mapsto \bar\sigma(t)$ is monotone increasing, cf. Proposition \ref{prop:mono} (ii), from \eqref{eq:lower_bound_sigma} we have $8\pi^2\bar\sigma(t)-1\geq \frac{1}{3}$ for all $t\in [0,t_0]$. Thus
\begin{align*}
\p_t\int_0^1|\tilde{\xi}|^2\leq -\frac{2}{3}\int_0^1 |\tilde{\xi}|^2, \quad t\in (0,t_0].
\end{align*}
This implies that $t\mapsto \|\tilde{\xi}(t,\cdot)\|_{L^2(\Sf)}$ is monotone decreasing. An integration in $t$ results in \eqref{eq:diff}, which completes the proof.
\end{proof}

The next lemma concerns the decay estimate of $t\mapsto \|\p_s\tilde{\xi}(t,\cdot)\|_{L^2(\Sf)}^2$.

\begin{lem}\label{lem:exp_deri}
Assume that $\bar\sigma(0)$ satisfy the same assumptions as in Lemma \ref{lem:exp_decay}. Assume furthermore that
\begin{equation}\label{eq:lower_bound_c}
c(0)^2\geq \frac{1}{2}.
\end{equation}
 Then
$t\mapsto \|\p_s\tilde{\xi}(t,\cdot)\|_{L^2(\Sf)}$ is monotone decreasing. Moreover,
 \begin{align*}
\|\p_s\tilde{\xi}(t,\cdot)\|_{L^2(\Sf)}\leq e^{-\frac{1}{16}t}\|\p_s\tilde{\xi}(0,\cdot)\|_{L^2(\Sf)} ,\quad t\in [0,t_0].
\end{align*}
\end{lem}

\begin{proof}
By \eqref{eq:tilde_equ1} and the expression of $\dot{c}$ in \eqref{eq:equ_eta},
\begin{align*}
\frac{1}{2}\p_t\int_0^1|\p_s\tilde{\xi}|^2 =-c(t)\dot{c}(t)=-\left(1-4\pi^2\bar\sigma\right)c^2 + 4\pi^2 c\int_0^1\sigma \p_sw_0\cdot \p_s\tilde{\xi}.
\end{align*}
Applying \eqref{eq:sqr2} to the last term in the above equation we obtain
\begin{equation}\label{eq:t_xi}
\begin{split}
\frac{1}{2}\p_t\int_0^1|\p_s\tilde{\xi}|^2&=-\left(1-4\pi^2\bar\sigma\right)c^2+4\pi^2\left(\frac{1-c^2}{2}\bar\sigma -\frac{1}{2}\int_0^1\sigma|\p_s\tilde{\xi}|^2\right)\\
&\leq -c^2+4\pi^2\bar\sigma c^2+4\pi^2\bar\sigma\frac{1-c^2}{2}.
\end{split}
\end{equation}

Next we note that
\begin{equation}\label{eq:Sigma}
4\pi^2\bar\sigma\leq 1-\frac{3}{4}\int_0^1|\p_s\tilde{\xi}|^2.
\end{equation}
Indeed, by (ii) of Proposition \ref{prop:mono} $\bar\sigma\leq \int_0^1|\xi|^2$, which in terms of $\tilde{\xi}$ and $c$ reads $\bar\sigma \leq \int_0^1|\tilde{\xi}|^2+\frac{c^2}{4\pi^2}$. By \eqref{eq:poincare}, $\bar\sigma\leq \frac{1}{16\pi^2}\int_0^1|\p_s\tilde{\xi}|^2 + \frac{c^2}{4\pi^2}$, which combined with \eqref{eq:tilde_equ1} gives \eqref{eq:Sigma}.\\
With \eqref{eq:Sigma} at hand, one can bound $\p_t\int_0^1|\p_s\tilde{\xi}|^2$ in terms of $\int_0^1|\p_s\tilde{\xi}|^2 ds$ as
\begin{equation}\label{eq:diff_inequ}
\frac{1}{2}\p_t\int_0^1|\p_s\tilde{\xi}|^2 ds\leq \int_0^1|\p_s\tilde{\xi}|^2 ds\left(-\frac{1}{4}+\frac{3}{8}\int_0^1|\p_s\tilde{\xi}|^2 ds\right).
\end{equation}
Indeed, since $c^2$ and $1-c^2$ are nonnegative (cf. \eqref{eq:tilde_equ1}), applying \eqref{eq:Sigma} to \eqref{eq:t_xi} we have that
\begin{align*}
\frac{1}{2}\p_t\int_0^1|\p_s\tilde{\xi}|^2\leq -c^2 + \left(1-\frac{3}{4}\int_0^1|\p_s\tilde{\xi}|^2\right)c^2 + \left(1-\frac{3}{4}\int_0^1|\p_s\tilde{\xi}|^2\right)\frac{1-c^2}{2}.
\end{align*}
Substituting $c^2$ by $1-\int_0^1|\p_s\tilde{\xi}|^2 ds$ (cf. \eqref{eq:tilde_equ1}) and after some algebraic manipulations we arrive at \eqref{eq:diff_inequ}.

Thus if $\int_0^1|\p_s\tilde{\xi}(0,s)|ds=1-c(0)^2\leq \frac{1}{2}$ by \eqref{eq:lower_bound_c}, then by \eqref{eq:diff_inequ} $\int_0^1|\p_s\tilde{\xi}(t,s)|ds\leq \frac{1}{2}$ for all $t\in [0,t_0]$. Moreover, $$\frac{1}{2}\p_t\int_0^1|\p_s\tilde{\xi}|^2 ds\leq -\frac{1}{16}\int_0^1|\p_s\tilde{\xi}|^2 ds$$ for all $t\in (0,t_0]$. Solving the above differential inequality we complete the proof.
\end{proof}

The exponential decay of $t\mapsto \|\p_s\tilde{\xi}(t,\cdot)\|_{L^2}$ immediately yields the convergence of the multiplicative factor $c(t)$ due to \eqref{eq:tilde_equ1}.

\begin{cor}\label{cor:c}
Under the same assumptions as in Lemma \ref{lem:exp_deri} we have
\begin{align*}
|c(t)^2-1|\leq e^{-\frac{1}{8}t}|c(0)^2-1| \text{ for all } t\in [0,t_0].
\end{align*}
\end{cor}

\begin{rmk}\label{rmk:differentiability}
The decay estimate on $c(t)$ implies that $t_0=T$, where $[0,t_0]$ is the interval of the definition of the solution to the ODE \eqref{eq:ode_ctheta} in Lemma \ref{lem:differentiability}. Indeed, since $c(t)$, $t\in [0,t_0]$, is bounded away from zero and has uniformly bounded modulus (cf. Corollary \ref{cor:c}), the solution to \eqref{eq:ode_ctheta} can be extended for longer time as long as $\|\xi(t,\cdot)\|_{C^{2+\alpha}(\Sf)}$ remains bounded.
\end{rmk}

In order to show the global existence we need to control the norm $\|\xi(t,s)\|_{2+\alpha,0}$ along the flow. For this it suffices to find a pointwise upper and lower bound on the ellipticity coefficient $\sigma(t,s)$. The next lemma says that if $\frac{1}{4\pi^2}-\bar\sigma$, $\|\tilde{\xi}\|_{L^2}$ and $\|\p_s\tilde{\xi}\|_{L^2}$ are sufficiently small at $t=0$, then the oscillation $\sigma(t,s)-\bar\sigma(t)$  and $|\xi|^2(t,s)-\int_0^1|\xi|^2 (t,s)ds$ are under control along the flow.

\begin{lem}\label{lem:para}
Given $\epsilon \in (0,\frac{1}{(32\pi^2)^2}]$, if
\begin{align}\label{eq:smallness}
\left(\frac{1}{4\pi^2}-\bar\sigma(0)\right)+
\left\|\tilde{\xi}(0,\cdot)\right\|^2_{L^2(\Sf)}+\left\|\p_s\tilde{\xi}(0,\cdot)\right\|^2_{L^2(\Sf)}\leq \epsilon,
\end{align}
then for all $(t,s)\in [0,T]\times \Sf$ we have
\begin{align*}
\left||\xi(t,s)|^2-\int_0^1|\xi(t,s)|^2ds \right|\leq 2\sqrt{\epsilon}, \quad  \left|\sigma(t,s)-\bar\sigma(t)\right|\leq 3\sqrt{\epsilon}.
\end{align*}
In particular, $\sigma$ satisfies
\begin{align*}
\frac{1}{4\pi^2}-4\sqrt{\epsilon}\leq \sigma(t,s)\leq \frac{1}{4\pi^2}+3\sqrt{\epsilon}\text{ for all }(t,s)\in [0,T]\times \Sf.
\end{align*}
\end{lem}
\begin{proof}
1. We observe that by the monotonicity properties from Proposition \ref{prop:mono}, Lemma \ref{lem:exp_decay} and Lemma \ref{lem:exp_deri} with $t_0=T$ (cf. Remark \ref{rmk:differentiability}), the smallness assumption \eqref{eq:smallness} indeed holds for all $t\in [0,T]$. Here we have used that \eqref{eq:smallness} implies \eqref{eq:lower_bound_c}.

2. We show that
\begin{align}\label{eq:ptxi}
\int_0^1|\p_t\xi(t,s)|^2 ds =\int_0^1|\xi(t,s)|^2 ds -\bar\sigma(t) \leq \epsilon \text{ for all } t\in (0,T].
\end{align}
Indeed, the first equality is due to \eqref{eq:identity1} and \eqref{eq:identity2}. To see the second inequality, we note that by the assumption \eqref{eq:smallness} and the observation above,
\begin{align*}
\int_0^1|\xi|^2 ds -\bar\sigma &= \left(\int_0^1|\xi|^2-\frac{1}{4\pi^2}\right)+\left(\frac{1}{4\pi^2}-\bar\sigma\right)\\
&=\int_0^1|\tilde{\xi}|^2 ds +(c^2-1)\frac{1}{4\pi^2}+\left(\frac{1}{4\pi^2}-\bar\sigma\right)\\
&=\int_0^1|\tilde{\xi}|^2 ds -\frac{1}{4\pi^2}\int_0^1|\p_s\tilde{\xi}|^2ds +\left(\frac{1}{4\pi^2}-\bar\sigma\right)\leq \epsilon,
\end{align*}
where in the second last inequality we have used \eqref{eq:tilde_equ1}.

3. We estimate the oscillation of $\sigma$ and $|\xi|^2$. First we show that the oscillation of $\sigma$ is bounded by the oscillation of $|\xi|^2$: for all $(t,s)\in [0,T]\times \Sf$
\begin{align}\label{eq:point_sigma}
\left|\sigma(t,s)-\bar\sigma(t)\right|\leq \sqrt{\epsilon}+\left|\frac{1}{2}|\xi(t,s)|^2-\frac{1}{2}\int_0^1|\xi(t,s)|^2ds\right|.
\end{align}
For this we multiply the equation of $\xi$ by $\p_s\xi$ and get $\p_t\xi\cdot \p_s\xi = \p_s\sigma + \xi\cdot \p_s\xi= \p_s\left(\sigma+\frac{1}{2}|\xi|^2\right)$. An integration in $s$ yields
\begin{align*}
\left|\sigma(t,s)+\frac{1}{2}|\xi(t,s)|^2 - \int_0^1\left(\sigma(t,s)+\frac{1}{2}|\xi(t,s)|^2\right) ds \right|\leq \int_0^1\left|\p_t\xi(t,s)\cdot \p_s\xi(t,s)\right| ds.
\end{align*}
By H\"older's inequality and \eqref{eq:ptxi},
\begin{align*}
\int_0^1\left|\p_t\xi(t,s)\cdot \p_s\xi(t,s)\right| ds \leq  \left(\int_0^1|\p_t\xi(t,s)|^2ds\right)^{1/2}\leq \sqrt{\epsilon}
\end{align*}
Combining the above two inequalities together we obtain \eqref{eq:point_sigma}.

Next we rewrite the oscillation of $|\xi|^2$ by using $|\tilde{\xi}|^2$:
\begin{align*}
\left|\frac{1}{2}|\xi(t,s)|^2-\frac{1}{2}\int_0^1|\xi(t,s)|^2ds\right|=\left|\frac{1}{2}|\tilde{\xi}(t,s)|^2-\frac{1}{2}\int_0^1|\tilde{\xi}(t,s)|^2ds+2c\tilde{\xi}\cdot w_0\right|.
\end{align*}
Using the fundamental theorem of calculus as well as the step 1, one can bound the oscillation of $|\tilde{\xi}|^2$ as
\begin{align*}
\left||\tilde{\xi}(t,s)|^2-\int_0^1|\tilde{\xi}(t,s)|^2ds \right|&\leq 2\int_0^1|\p_s\tilde{\xi}\cdot\tilde{\xi}| ds\leq 2\|\p_s\tilde{\xi}(t,\cdot)\|_{L^2}\|\tilde{\xi}(t,\cdot)\|_{L^2}\leq 2\epsilon.
\end{align*}
For the term $c\tilde{\xi}\cdot w_0$ we use $|c|\leq 1$, $|w_0|=\frac{1}{2\pi}$ and the above oscillation estimate to get
\begin{equation*}
|c\tilde{\xi}\cdot w_0|\leq \frac{1}{2\pi}|\tilde{\xi}(t,s)| \leq \frac{1}{2\pi}\left(\int_0^1|\tilde{\xi}(t,s)|^2ds+2\epsilon\right)^{1/2}\leq \frac{1}{2\pi}\sqrt{3\epsilon}.
\end{equation*}
Combining together we have for all $(t,s)\in [0,T]\times \Sf$
\begin{align*}
\left|\frac{1}{2}|\xi(t,s)|^2-\frac{1}{2}\int_0^1|\xi(t,s)|^2ds\right|&\leq \epsilon+2\frac{1}{2\pi}\sqrt{3\epsilon}\leq 2\sqrt{\epsilon}.
\end{align*}
Combining this with \eqref{eq:point_sigma} we obtain the desired oscillation estimate for $\sigma$.
Taking $\epsilon \leq \frac{1}{(32\pi^2)^2}$ and using the step 1 we obtain
the pointwise lower bound for $\sigma(t,s)$.
The upper bound follows from \eqref{eq:point_sigma} and Proposition \ref{prop:mono}.
\end{proof}

At the end of this section we show the global well-posedness of the normalized flow in the H\"older class $C^{2+\alpha,\alpha/2}$, under the assumption that initially the curve is sufficiently close to the stationary solution $w_0$.

\begin{thm}\label{thm:global}
Let $w_0$ be the stationary solution defined in \eqref{eq:w0}. Given an initial datum $\xi_0\in \tilde{\mathcal{A}}\cap C^{2+\alpha}(\Sf)$, which satisfies
\begin{align}\label{eq:smallness2}
\|\xi_0-w_0\|_{H^2(\Sf)}\leq \epsilon_0
\end{align}
for some small universal constant $\epsilon_0>0$,  the Cauchy problem \eqref{eq:wellposed} has a solution $$\xi\in C^{2+\alpha, \alpha/2}([0,\infty)\times \Sf), \p_t\xi\in C^{\alpha,\alpha/2}((0,\infty)\times \Sf).$$
\end{thm}

\begin{proof}
Let $\xi\in C^{2+\alpha,\alpha/2}([0,T]\times \Sf)$, $T\in (0,1)$ depending on $\|\xi_0\|_{C^{2+\alpha}(\Sf)}$, be the classical solution to the Cauchy problem \eqref{eq:wellposed} with the initial datum $\xi_0$ (cf. Theorem \ref{thm:wellposedness}).  Let $\sigma$ be the Lagrange multiplier such that $\xi_0+\p_s(\sigma\p_s\xi_0)\in T_{\xi_0}\tilde{\mathcal{A}}$. Similar arguments as in Lemma \ref{lem:proj} yield
\begin{align}\label{eq:sig00}
\p_{ss}\sigma-|\p_{ss}\xi_0|^2\sigma=-1.
\end{align}
Note that from \eqref{eq:sig00} we have
$$\int_0^1\sigma=\|\p_s(\sigma\p_s\xi_0)\|_{L^2}^2\leq \|\xi_0\|_{L^2}^2,$$
which is no larger than $\frac{1}{4\pi^2}$ by Wirtinger's inequality. We want to show if $\epsilon_0$ is sufficiently small, then the smallness assumption \eqref{eq:smallness} from Lemma \ref{lem:para} is satisfied.

To see this, we first consider $\tilde{\xi}_0:=\xi_0-cw_0(\cdot+\theta)$, where $cw_0(\cdot+\theta)\in \argmin_{w\in \mathcal{C}} \|\xi_0-w\|_{L^2}$. Then \eqref{eq:smallness2} implies
$$\|\tilde{\xi}_0\|_{L^2(\Sf)}\leq \epsilon_0\text{ and }\|\p_s\tilde{\xi}_0\|_{L^2(\Sf)}\leq 2\pi\sqrt{\epsilon_0(1+\pi^{-1})}.$$
Indeed, the first inequality is immediate since by the definition of the projection $\|\tilde{\xi}_0\|_{L^2(\Sf)}\leq \|\xi-w_0\|_{L^2(\Sf)}$. The second inequality follows from \eqref{eq:tilde_equ1} having observed that \begin{multline*}\frac{c^2}{4\pi^2}=\|\xi_0\|_{L^2(\Sf)}^2-\|\tilde{\xi}_0\|_{L^2(\Sf)}^2\geq
\|w_0\|_{L^2(\Sf)}^2-2\|w_0\|_{L^2(\Sf)}\|\xi-w_0\|_{L^2(\Sf)}-\|\tilde{\xi}_0\|_{L^2(\Sf)}^2\\ \geq \frac 1 {4\pi^2}- \epsilon_0(1+\pi^{-1}).\end{multline*} Next, by viewing \eqref{eq:sig00} as a perturbation of
$$\p_{ss}\sigma-|\p_{ss}w_0|^2\sigma=\p_{ss}\sigma-4\pi^2\sigma=-1$$
 with periodic boundary conditions, whose solution is the constant function $\frac{1}{4\pi^2}$, we have that the solution to \eqref{eq:sig00} satisfies
\begin{align*}
\left|\frac{1}{4\pi^2}-\int_0^1\sigma ds\right|\leq C \epsilon_0
\end{align*}
for some universal constant $C>0$. Hence if $\epsilon_0$ is sufficiently small but universal, the smallness assumption \eqref{eq:smallness} holds.

Now we apply Lemma \ref{lem:para} to $\xi$ and get
\begin{equation}\label{eq:sig000}
\frac{1}{8\pi^2}\leq \sigma(t,s)\leq \frac{1}{2\pi^2}
\end{equation}
in $[0,T]\times \Sf$.
%Thus the equation of $\xi$ is parabolic with the parabolicity bounded uniformly from above and below.
Furthermore, for each fixed $t$, $\sigma(t,\cdot)$ is Lipschitz continuous with uniformly bounded Lipschitz constant. Indeed, an integration of the equation of $\sigma$ (cf. \eqref{eq:eta_sigma}) yields $\int_{\Sf} \sigma |\p_{ss}\eta|^2 =1$, which implies that $\int_{\Sf}|\p_{ss}\sigma(t,s)|ds\leq 2$. By the regularity theory for the parabolic equations, cf. \cite{lieberman1992},
\begin{align}\label{eq:schauder}
\sup_{t\in [T/2,T]} \|\xi(t,\cdot)\|_{C^{2+\alpha}(\Sf)}\leq C (T^{-(2+\alpha)}+1)\|\xi\|_{L^\infty([T/4,T]\times \Sf)}
\end{align}
for some $C>0$ only depending on $d$ and $\alpha$ (this is the interior Schauder estimate for parabolic equations. Since our solutions satisfy periodic bound conditions, the estimate holds globally in $\Sf$). In the meanwhile, the smallness assumption \eqref{eq:smallness} of Lemma \ref{lem:para} is satisfied for all $t\in [0,T]$ due to the monotonicity properties of $\int_0^1\sigma(t,s)ds$, $\|\tilde{\xi}(t,\cdot)\|_{L^2(\Sf)}$ and $\|\p_s\tilde{\xi}(t,\cdot)\|_{L^2(\Sf)}$ (Proposition \ref{prop:mono} (ii), Lemma \ref{lem:exp_decay} and Lemma \ref{lem:exp_deri}). In particular,
\begin{equation}\label{eq:Linfty_xi}
\|\xi\|_{L^\infty([T/4,T]\times\Sf)}\leq \|cw_0\|_{L^\infty([T/4,T]\times\Sf)}+ \|\tilde{\xi}\|_{L^\infty([T/4,T]\times\Sf)}\leq \bar C,
\end{equation}
where $\bar C$ is an absolute constant.
%Thus the members of \eqref{eq:schauder} are bounded by constants independent of $t_0$.
We apply Theorem \ref{thm:wellposedness} starting from $t=T$. The solution $\xi$ extends in the H\"older class $C^{2+\alpha,\alpha/2}$ up to $t=T+T_0$ for some $T_0>0$. Lemma \ref{lem:para} applies and yields \eqref{eq:sig000} and \eqref{eq:Linfty_xi} in $[0,T+T_0]\times \Sf$. Again by the interior Schauder estimate
\begin{equation*}
\sup_{t\in [\frac{T+T_0}{2},T+T_0]} \|\xi(t,\cdot)\|_{C^{2+\alpha}(\Sf)}\leq C ((T+T_0)^{-(2+\alpha)}+1)\|\xi\|_{L^\infty([0,T+T_0]\times \Sf)}.
\end{equation*}
Since $\|\xi(t,\cdot)\|_{C^{2+\alpha}(\Sf)}$ cannot blow up as time getting large, the time steps have a uniform lower bound $\tilde{T}_0>0$. Repeating the above arguments from $T+2\tilde{T}_0$, $T+3\tilde{T}_0$ and so on, we obtain the global existence of the Cauchy problem \eqref{eq:wellposed} in the H\"older class $C^{2+\alpha,\alpha/2}$.
\end{proof}

As a by product we also obtain the exponential decay of our solution to the stationary solution $w_0$ under the initial smallness assumption.

\begin{thm}\label{thm:exp}
Under the same assumptions as in Theorem \ref{thm:global} we have for all $t>0$,
\begin{align*}
\|\xi(t,s)-w_0(s+\theta_\infty)\|_{L^\infty(\Sf)}\leq C e^{-t/16}\|\xi_0-w_0\|_{H^2(\Sf)}
\end{align*}
for some universal constant $C>0$ and some constant $\theta_\infty$.
\end{thm}
\begin{proof}
By Lemma \ref{lem:exp_decay}, Lemma \ref{lem:para} and Corollary \ref{cor:c}, if $\epsilon_0:=\|\xi_0-w_0\|_{H^2(\Sf)}$ is sufficiently small (say $\epsilon_0\leq \frac{1}{(100\pi^2)^2}$), then for all $t>0$
\begin{align*}
\|\tilde{\xi}(t,\cdot)\|_{L^\infty(\Sf)}+(1-c^2(t))\leq Ce^{-\frac{1}{16}t}\epsilon_0
\end{align*}
for some absolute constant $C>0$. We still need to estimate the evolution of $\theta(t)$. Note that from the expression of $\dot{\theta}$ in \eqref{eq:equ_eta}
\begin{align*}
|\dot{\theta}(t)|\leq \frac{4\pi^2}{c(t)}\frac{1}{2\pi}\max_{s\in \Sf}\sigma(t,s)\|\p_s\tilde{\xi}(t,\cdot)\|_{L^2(\Sf)},\quad t>0.
\end{align*}
By \eqref{eq:point_sigma} if $\epsilon_0$ in \eqref{eq:smallness2} is sufficiently small, then $\sigma(t,s)\leq \int_0^1\sigma(t,s) ds + 3\sqrt{\epsilon_0}\leq \frac{1}{2\pi^2}$.  Thus combining Corollary \ref{cor:c} and Lemma \ref{lem:exp_deri} we infer
\begin{align*}
|\dot{\theta}(t)|\leq e^{-\frac{1}{16}t}\|\p_s\tilde{\xi}(0,\cdot)\|_{L^2(\Sf)},\quad t>0.
\end{align*}
This implies that $\lim_{t\rightarrow \infty}\theta(t)$ exists. Letting $\theta_\infty$ denote the limit we obtain
$$|\theta(t)-\theta_\infty|\leq 16e^{-\frac{1}{16}t}\|\p_s\tilde{\xi}(0,\cdot)\|_{L^2}\leq Ce^{-\frac{1}{16}t}\epsilon_0$$
after an integration from $t$ to $\infty$. Thus by the triangle inequality
\begin{align*}
\left|\xi(t,s)-w_0(s+\theta_\infty)\right|&\leq \left|\xi(t,s)-c(t)w_0(s+\theta(t))\right|+\left|(c(t)-1)w_0(s+\theta(t))\right|\\
&+\left|w_0(s+\theta(t))-w_0(s+\theta_\infty)\right|\\
&\leq |\tilde{\xi}(t,s)|+\frac{1}{2\pi}|c(t)-1|+ |\theta(t)-\theta_\infty|\\
&\leq Ce^{-\frac{1}{16}t}\epsilon_0.
\end{align*}
\end{proof}

\begin{rmk} The Hessian \eqref{Hes}  of the $L^2$-mass is strictly negative-definite on the tangent vectors
$\tilde\xi\in T_{w_0}\tilde{\mathcal A}$
which satisfy \eqref{eq:poincare}, i.e., on those which are orthogonal to the pure rotations:
\be\langle Hess\; M(w_0) \tilde\xi, \tilde\xi
 \rangle_{T_{w_0}\tilde{\mathcal{A}}}=\int_0^1 |\tilde\xi|^2(s)\, ds-\int_0^1 \varsigma(s) \,
ds\leq  -3\int_0^1  |\tilde\xi|^2(s)\, ds=-3 \langle \tilde\xi, \tilde\xi
  \rangle_{T_{w_0}\tilde{\mathcal{A}}}.\label{Hes1}\ee
 Here $\varsigma$ is the initial tension of a geodesic emanating from $w_0$ at the direction $\tilde\xi$, which satisfies \be\label{eq:geod1} \p_{ss}\varsigma- |\p_{ss} w_0|^2 \varsigma+|\p_s\tilde\xi|^2=0 \ee (see \cite{Preston-2011,SV17}). Indeed, since $|\p_{ss} w_0|\equiv 2\pi$, an integration of \eqref{eq:geod1} together with \eqref{eq:poincare} implies
 $$\int_0^1 \varsigma(s) \,
ds = \frac 1 {4\pi^2} \int_0^1 |\p_s\tilde\xi|^2(s)\, ds \geq 4 \int_0^1 |\tilde\xi|^2(s)\, ds,$$
which gives \eqref{Hes1}. Then one can anticipate the exponential decay (Theorem \ref{thm:exp}) of the gradient flow in a neighbourhood of $w_0$  via a Bakry-\'Emery argument, cf. \cite{villani03topics}. However, such argument is not applicablie in our situation since the Riemannian connection of $\tilde{\mathcal{A}}$ is not smooth and $(\tilde{\mathcal{A}}, d_{\tilde{A}})$ is not a geodesic metric space, cf. Theorem 4.2 in \cite{Preston-2012} and \cite{Mol17,BMM16}. \end{rmk}

\subsection{Global existence without restrictions on the initial data}  \label{sec:weak} We conclude by showing global solvability of the normalized flow in a generalized sense without any restrictions on the initial data. It is an adaptation of the approach we recently developed in \cite{ShV17} for a different gradient flow.

We begin by rewriting our flow in a form which explicitly involves the arc length parametrization constraint (cf. the beginning of the Section \ref{sec:normalize_flow}):
\begin{equation}\label{eq:weak}
\begin{split}
\begin{cases}
\p_t\xi&= \p_s(\sigma \p_s\xi) + \xi \\
|\p_s\xi| & =1
\end{cases}
\text{ for  } (t,s)\in Q_\infty:=(0,\infty)\times \Sf.
\end{split}
\end{equation}

\begin{defi}\label{defi:weak_sol}
Given an initial datum $\xi_0\in W^{1,\infty}(\Sf;\R^d)$ with  $|\p_s\xi_0(s)|\leq 1$ for a.e. $s\in \mathcal S^1$, we call a pair $(\xi,\sigma)$ a generalized solution to the normalized UCMCF if the following hold
\begin{itemize}
\item[(i)]  $\xi\in L^\infty_{loc}([0,\infty); W^{1,\infty}(\mathcal S^1))^d$, $\p_t\xi\in L^2_{loc}([0,\infty);L^2(\mathcal S^1))^d$, $\sigma\in L^2_{loc}([0,\infty); H^1(\mathcal S^1))$ and $\sigma\p_s\xi\in L^2_{loc}([0,\infty); H^1(\mathcal S^1))^d$.
\item[(ii)] The pair $(\xi, \sigma)$ satisfies
for a.e. $(t,s)\in Q_\infty$
\begin{align} \p_t\xi(t,s) = \p_s(\sigma(t,s) \p_s\xi(t,s)) &+ \xi,\label{eq:defii1}\\
\sigma(t,s) \left(|\p_s\xi(t,s)|^2 -1\right ) &=0,\label{eq:defii}\\
|\p_s\xi(t,s)|&\leq 1,\label{eq:defii2}
\end{align}
and the initial condition
\begin{align*}
\xi(0,s)=\xi_0(s) \text{ for a.e. }s.
\end{align*}
\item[(iii)] The solution $\xi$ satisfies the energy dissipation inequality
\begin{equation} \label{e:edeq}
\int_{\mathcal S^1} |\p_t\xi(t,s)|^2 ds\leq \int_{\mathcal S^1} \xi\cdot \p_t\xi(t,s) ds
\end{equation} for a.e. $t\in (0,\infty)$.
\end{itemize}
\end{defi}

\begin{rmk}[Strong and weak constraint]
The generalized solutions in Definition \ref{defi:weak_sol} are not required to satisfy the strong constraint  $|\p_s\xi|= 1$ but merely the relaxed one
$$
\sigma \left(|\p_s\xi|^2 -1\right ) =0,\quad  |\p_s\xi|\leq 1.
$$
In the next remark we will show that under the regularity assumptions $\xi\in C^1(\overline{Q_\infty})\cap C^2(Q_\infty)$ and $|\p_s\xi_0|=1$, the generalized solutions solve \eqref{eq:weak} in the classical sense. However, without the regularity assumptions we do not know whether the constraint $|\p_s\xi|=1$ is satisfied or not. \end{rmk}

\begin{rmk}[Relation with the classical solution]
\label{rmk:gene_sol}
It is not hard to see that if $(\xi,\sigma)$ is a $C^2$ regular solution to  \eqref{eq:weak}, then it is also a generalized solution in the sense of Definition \ref{defi:weak_sol}; in particular, \eqref{eq:defii2}
 and \eqref{e:edeq} become strict equalities. On the other hand we claim that any  generalized solution $(\xi,\sigma)$ with $\xi\in C^1(\overline{Q_\infty})\cap C^2(Q_\infty)$ and $|\p_s\xi_0|=1$ solves \eqref{eq:weak}.

Formally, this claim is a trivial consequence of \eqref{eq:defii} since $\sigma$ is expected to be strictly positive by the strong maximum principle, cf. Remark \ref{rmk:sign_sigma}. However, we cannot guarantee the strict positivity of $\sigma$ for the generalized solutions. Nevertheless, to prove the claim it suffices to show that the open set $U:=\{(t,s)\in Q_\infty: |\p_s\xi(t,s)|<1\}$ is empty. Suppose not, then $\sigma=0$ a.e. in $U$ due to  \eqref{eq:defii}. This implies that $\p_t\xi=\xi$ in $U$, whence $\p_t(|\p_s\xi|^2)=2 |\p_s\xi|^2$. For each $(t_0,s_0)\in U$, let $t_1=\inf\{t\geq 0: (t,t_0)\times\{ s_0\}\subset U\}$.
If $t_1=0$ then
\begin{equation}\label{e:ca2}
|\p_s\xi(t_1,s_0)|=1
\end{equation}
due to our assumption about $\xi_0$, and if $t_1>0$ then \eqref{e:ca2} also holds by the continuity of $\p_s\xi$. From $\p_t(|\p_s\xi|^2)\geq 0$ in $U$ we immediately deduce that \[|\p_s\xi(t_0,s_0)|^2\geq |\p_s\xi(t_1,s_0)|^2= 1,\] arriving at a contradiction.
\end{rmk}

The next theorem concerns the global existence of the generalized solution without any smallness or closeness assumption on the initial datum. We stress that the theorem does not cover Theorem \ref{thm:global} due to the relaxation of the unit speed constraint in \eqref{eq:weak}.

\begin{thm}\label{thm:globex} For every $\xi_0\in W^{1,\infty}(\Sf;\R^d)$ with $|\p_s\xi_0(s)|\leq 1$ for a.e. $s\in \Sf$, there exists a (global in time) generalized solution $(\xi,\sigma)$ to the normalized UCMCF, and  $\sigma(t,s)\geq 0$ for almost every $(t,s)\in Q_\infty$.
\end{thm}

The proof mimicks the one of \cite[Theorem 3]{ShV17} and has the following outline. We rewrite \eqref{eq:weak} as a first-order system, and approximate it by Hilbertian gradient flows. Let $\kappa:=\sigma \p_s\xi$, then the problem in the new variables $(\xi, \kappa, \sigma)$ would read
\begin{equation}\label{eq:whip2}
\begin{split}
\begin{cases}
\p_t\xi=\p_s\kappa+\xi\\
\kappa=\sigma\p_s\xi\\
\sigma=\kappa\cdot\p_s\xi.
\end{cases}
\end{split}
\end{equation}
For $\epsilon>0$, let
\begin{equation*}\label{eq:F_eps}
F^\epsilon:\R^d\rightarrow \R^d,\quad F^\epsilon(\kappa):=\epsilon\kappa+\frac{\kappa}{\sqrt{\epsilon+|\kappa|^2}},
\end{equation*}
$$G^\epsilon(\tau):=(F^\epsilon)^{-1}(\tau),$$
and consider the problem
\begin{equation}\label{eq:approx_whip}
\p_t\xi^\epsilon= \p_s(G^\epsilon(\p_s\xi^\epsilon))+\xi^\epsilon
\text{ in } Q_\infty.
\end{equation}
Let us introduce the functional
\begin{equation*}\label{eq:perturb_energy}
\mathcal{E}^\epsilon(\xi):=\int_{\Sf}\epsilon \left(\frac{|G^\epsilon(\p_s \xi)|^2}{2}-\frac{1}{\sqrt{\epsilon+|G^\epsilon(\p_s \xi)|^2}}\right)-\frac 1 2 |\xi|^2.
\end{equation*}
Then \eqref{eq:approx_whip} can be interpreted as a negative gradient flow of $\mathcal{E}^\epsilon$ with respect to the flat Hilbertian structure of $L^2(\Sf;\R^d)$.
The existence of a unique smooth solution $\xi^\epsilon: [0,\infty)\times [0,1]\rightarrow \R^d$ to \eqref{eq:approx_whip} follows from Amann's theory \cite{Amann}. The solutions satisfy uniform energy estimates as in \cite[Proposition 3.1]{ShV17}. Moreover, $\partial_s \xi^\epsilon$ has a uniform $L^\infty$ bound as in \cite[Proposition 3.3]{ShV17}.
We now set
\begin{equation*}\label{eq:kappa_sigma}
\begin{split}
\kappa^\epsilon := G^\epsilon(\p_s\xi^\epsilon),\quad
\sigma^\epsilon := G^\epsilon(\p_s\xi^\epsilon)\cdot \p_s\xi^\epsilon\geq 0.
\end{split}
\end{equation*}
Arguing as in the proof of  \cite[Proposition 3.4 and Theorem 3]{ShV17}, we can pass to the limit and obtain a solution $(\xi, \kappa, \sigma)$ to \eqref{eq:whip2}. The pair $(\xi, \sigma)$ solves the normalized UCMCF in the sense of Definition \ref{defi:weak_sol}. We refer to \cite{ShV17} for the full implementation.

\begin{rmk} \label{rmkback} One can adapt the approach of \cite[Section 6]{ShV17} to construct backward generalized solutions to the normalized UCMCF. It seems however that all one can get in this way is the trivial solution $(\xi,\sigma)(t)=(e^t \xi_0,0)$, $t\leq 0$. It satisfies \eqref{eq:defii1}--\eqref{e:edeq}, and is smooth provided $\xi_0$ is smooth, but obviously violates the strong constraint $|\p_s\xi|=1$. This contrasts with Remark \ref{rmk:gene_sol} and with \cite{ShV17} where smoothness implied the strong constraint. Consequently, the method of \cite[Section 6]{ShV17} for constructing two different solutions emanating from an initial datum $\xi_0$ with $|\p_s\xi_0|=1$ is not applicable. This leads us to conjecture the uniqueness of the generalized solutions to the normalized UCMCF.
\end{rmk}
\appendix
\section{Higher dimensional UCMCF} \label{higherdim}
In this Appendix we describe how our approach can be implemented in the case of evolution of surfaces. For the sake of simplicity of the presentation, we  work with embeddings of a compact manifold into the ambient space $\R^d$, but this can be generalized in various directions, in particular, one can consider immersions instead of embeddings.

\subsection{Riemannian structure} Fix a smooth compact connected $k$-dimensional submanifold $\mathcal M$ of $\R^d$.
Without loss of generality in the sequel we assume that $\vol(\mathcal{M})=1$.
 Let $\mathcal{K}_k$ be the space of $H^m$-regular embeddings $\eta: \mathcal M \to \R^d$, $\int_{\mathcal M} \eta\ d\mathcal H^k =0$, $m> \frac {n +2}2 $. Each element $v\in T_\eta\mathcal{K}_k$ can be identified with a vector field $v:\mathcal M \to \R^d$. We endow the space ${\mathcal{K}}_k$ with the parametrization-invariant $L^2$ Riemannian metric (cf. \cite{MM5})
\be
\langle v,w\rangle_{T_\eta\mathcal{K}_k}:=\int_{\eta(\mathcal{M})} (v\cdot w)\circ (\eta^{-1})\ d \mathcal H^{k}=\int_{\mathcal{M}} v\cdot w\ J\eta \ d\mathcal{H}^k, \label{eq:Rm}
\ee
which has a degenerate Riemannian distance \cite{BHM12}. Here $J\eta(x):=\sqrt{\det(d\eta_x)^\ast \circ (d\eta_x)}$ for each $x\in \mathcal{M}$ is the Jacobian of $\eta$.
Let $$\vol:\mathcal{K}_k\rightarrow \R, \ \vol(\eta)=\mathcal H^k({\eta(\mathcal M)} )$$ be the volume functional. By Sobolev embedding $H^{m}\subset C^1$, $\vol$ is continuous w.r.t. the $H^{m}$-topology of $\mathcal K_k$.
Sometimes we will also use a flat metric $\langle \cdot, \cdot \rangle^\ast$ on $\mathcal{K}_k$:\begin{align}\label{eq:metrickkf}
\langle v,w\rangle_{T_\eta \mathcal{K}_k}^*:=\int_{\mathcal M} v\cdot w\, d \mathcal H^k.
\end{align}

We consider the submanifold of $\mathcal K_k$ consisting of uniformly dilating embeddings, i.e.,
\begin{align*}
\mathcal{A}_k:=\{\eta\in \mathcal{K}_k: \vol(\eta)>0, \ \eta_\#(\mathcal H^{k}\mres \mathcal M)= \frac 1 {\vol(\eta)}\mathcal H^{k}\mres \eta(\mathcal M)\}.
\end{align*}
 Let us also define the submanifold of $\mathcal{A}_k$ consisting of volume-preserving embeddings
\begin{align} \label{volpres}
\tilde{\mathcal{A}}_k:=\{\eta\in \mathcal{A}_k, \ \vol(\eta)=1\}.
\end{align}
The tangent space at $\eta\in \tilde{\mathcal{A}}_k$ is
\begin{align*}
T_\eta\tilde{\mathcal{A}}_k=\{h\in T_\eta \mathcal{K}_k: \di _{\eta(\mathcal{M})} (h\circ \eta^{-1})=0\},
\end{align*}
thus it is not hard to verify that $T_\eta\mathcal{A}_k=\{h\in T_\eta \mathcal{K}_k: \di _{\eta(\mathcal{M})} (h\circ \eta^{-1})=const\}$.
The metrics \eqref{eq:Rm} and \eqref{eq:metrickkf} induce metrics on $\mathcal{A}_k$: for $\eta\in\mathcal{A}_k$
\begin{align}\label{eq:metrick}
\langle v,w\rangle_{T_\eta\mathcal{A}_k}=\int_{\mathcal M} v\cdot w\, \vol(\eta)\, d \mathcal H^k,
\end{align}
\begin{align}\label{eq:metrickf}
\langle v,w\rangle_{T_\eta\mathcal{A}_k}^*=\int_{\mathcal M} v\cdot w\, d \mathcal H^k.
\end{align}
The induced Riemannian metric on $\tilde{\mathcal{A}}_k$ (both from $\langle\cdot, \cdot\rangle$ and $\langle \cdot, \cdot \rangle^\ast$) is then \begin{align}\label{eq:metrickvp}
\langle v,w\rangle_{T_\eta\tilde{\mathcal{A}}_k}=\langle v,w \rangle^\ast_{T_\eta\tilde{\mathcal{A}}_k}=\int_{\mathcal M} v\cdot w\, d \mathcal H^k.
\end{align}

\subsection{Orthogonal projection} By \cite[Theorem 7]{BMM16}, the orthogonal projection $\tilde P_\eta:T_\eta \mathcal K_k\to T_\eta \tilde{\mathcal A}_k$ (with respect to the invariant $L^2$ metric $\langle\cdot, \cdot\rangle_{T_\eta\mathcal{K}_k}$) is
\be
\tilde P_\eta(z)=z-\sigma \vec H(\eta)- d \eta_\cdot \grad^g\sigma,\label{eq.projnor}
\ee
where $\vec H(\eta)$ is the vectorial mean curvature (the trace of the second fundamental form) corresponding to the embedding $\eta$, the metric $g=\eta^\ast \cdot$ is the pull-back of the Euclidean metric $\cdot$ on $\R^d$,
 and $\sigma:\mathcal M \to \R$ is a Lagrange multiplier, which is a solution to
 \begin{align*}
 \Delta_g \sigma -|\vec H(\eta)|^2 \sigma =\di_{\eta(\mathcal{M})}(z\circ\eta^{-1}).
 \end{align*}
Employing the methods of \cite{BMM16,Mol17}, one can derive from \eqref{eq.projnor} that the map $P_\eta: T_\eta\mathcal{K}_k\rightarrow T_\eta\mathcal{A}_k$,
\begin{equation}\label{eq.proj1k}
\begin{split}
P_\eta(z)=z-\sigma \vec H(\eta)- d \eta_\cdot \grad^g\sigma, \text{ where }\\
\Delta_g \sigma -|\vec H(\eta)|^2 \sigma =\di_{\eta(\mathcal{M})}(z\circ \eta^{-1})+const,\quad \int_{\mathcal M}\sigma\,d\mathcal H^k=0,
\end{split}
\end{equation}
is an orthogonal projection.
The key observations in the proof of this claim are that the volume density for $\eta\in \mathcal{A}_k$ is constant (equal to $\vol(\eta)$) and the identity following from the divergence formula
\begin{equation}\label{key}
\begin{split}
\int_{\mathcal M}w\cdot\left(\sigma\vec H(\eta)+ d \eta_\cdot \grad^g \sigma\right) \vol(\eta)\, d\mathcal H^k&=-\int_{\eta(\mathcal M)}(\sigma \circ\eta^{-1}) \di_{\eta(\mathcal{M})}(w\circ \eta^{-1}) \,d\mathcal H^k\\
&=-const\int_{\mathcal M}\sigma \vol(\eta)\,d\mathcal H^k
\end{split}
\end{equation}
for any $w\in T_\eta\mathcal{A}_k$. Here in the last equality we have used $w\in T_\eta\mathcal{A}_k$ and the characterization of the tangent space $T_\eta\mathcal{A}_k$.

\subsection{The gradient flow} The UCMCF is the gradient flow \be\partial_t \eta=- \grad_{\mathcal A_k} \vol(\eta) \label{eq:gf}\ee of the volume functional on the space $\mathcal{A}_k$ under the metric \eqref{eq:metrick}. By construction, the flow operator $$T_t:\eta(0,s)\mapsto \eta(t,s), \ s\in \mathcal M,$$ complies with \eqref{good}.

By the first variation of area formula, the negative $\mathcal K_k$-gradient of the volume functional is simply $\vec H$, and \be \langle \vec H (\eta), \eta\rangle_{T_\eta\mathcal{K}_k}=-k\vol(\eta).\label{heta}\ee With the projection \eqref{eq.proj1k} at hand, by an argument similar to the one from the Section \ref{1st}, we can express the gradient flow \eqref{eq:gf} in the form \be \label{eq:gf2}
\p_t\eta=\tilde\sigma \vec H (\eta)+ d \eta_\cdot \grad^g\tilde\sigma, \quad
\int_{\mathcal M}\tilde\sigma\,d\mathcal H^k=1,\quad \eta\in \mathcal A_k.
\ee

A direct computation yields that if a pair $(\eta(t),\tilde\sigma(t))$ solves \eqref{eq:gf2}, and $r:\mathcal M \to \mathcal M$ is a volume-preserving diffeomorphism, then $(\eta(t)\circ r,\tilde\sigma(t)\circ r)$ also solves \eqref{eq:gf2}.  Note that the reparametrizations $r$ which do not preserve the Hausdorff measure on $\mathcal M$ are ruled out automatically by our construction. Thus UCMCF is a truly geometric flow since it does not depend on possible reparametrizations of the evolving submanifold $\eta(t)(\mathcal M)$. This claim will become completely transparent after we recast our flow into a parametrization-free form \eqref{eq:gf5}.

Let $$M(\eta):=\frac 12 \int_{\mathcal M} |\eta|^2\, d \mathcal H^k$$
be the $L^2$-mass functional.
We are going to see that this functional decays with a constant speed along our gradient flow, cf. Proposition \ref{prop:L2}.
Indeed, since $\eta\in T_\eta\mathcal{A}_k$ (since $\di_{\eta(\mathcal{M})}(\eta\circ \eta^{-1})=k$),
\begin{equation*}
\begin{split}
 \partial_t M(\eta) &= \int_{\mathcal M} \eta\cdot \partial_t \eta\, d \mathcal H^k= \frac 1 {\vol (\eta)}  \langle \eta, \partial_t\eta\rangle_{T_\eta\mathcal{K}_k}= \frac 1 {\vol (\eta)}  \langle \eta, \partial_t\eta\rangle_{T_\eta\mathcal{A}_k}\\ &=-\frac 1 {\vol (\eta)}  \langle \eta, \grad_{\mathcal A_k} \vol(\eta) \rangle_{T_\eta\mathcal{A}_k}=-\frac 1 {\vol (\eta)} \langle \eta, \grad_{\mathcal K_k} \vol(\eta) \rangle_{T_\eta\mathcal{K}_k} \\
&=\frac 1 {\vol (\eta)} \langle \eta, \vec H(\eta) \rangle_{T_\eta\mathcal{K}_k} =-k
\end{split}
\end{equation*}
by \eqref{heta}. Thus, our flow collapses in finite time $t_*=\frac{1}{k}M(\eta_0)$.

\begin{rmk} \label{rimnond}  The Riemannian distances $d_{\mathcal{A}^*_k}$
and $d_{\tilde{\mathcal{A}}_k}$ on submanifolds of $\mathcal{K}_k$
are non-degenerate since they are controlled from below by the Hilbertian distance $d_{\mathcal{K}^*_k}$ (which is induced by the Riemannian metric $\langle \cdot, \cdot\rangle ^\ast$ in \eqref{eq:metrickf}). We do not know whether the Riemannian distance $d_{\mathcal{A}_k}$ is non-degenerate in general, but as we observed in Proposition \ref{rimnond1}, the conjecture is true for $k=1$. Nevertheless, we have got backup ways to render UCMCF as a gradient flow with respect to a non-degenerate distance. Namely, the gradient flow
\be\partial_t \eta=- \grad_{\mathcal A^*_k} \left(\ln \vol(\eta)\right) \label{eq:gfr}\ee
reproduces \eqref{eq:gf}. This is immediate by observing that $$\grad_{\mathcal A_k}  \vol(\eta)=\frac 1 {\vol (\eta)}\grad_{\mathcal A^*_k}  \vol(\eta)= \grad_{\mathcal A^*_k} \left(\ln \vol(\eta)\right).$$ Other options are presented in Appendix \ref{sec:mult}. \end{rmk}

\section{Normalized flow}\label{sec:mult}
The general higher dimensional UCMCF \eqref{eq:gf2} can be renormalized in the same way as the curve-shortening flow. Namely, the new time is
\begin{align*}
\tau(t):=-\ln \vol(\eta(t)),
\end{align*}
and the new unknown functions are
$$
\xi(\tau, s ):=\frac{\eta(t(\tau),s)}{\vol(\eta(t(\tau)))},\
\sigma(\tau, s)=\frac{\tilde{\sigma}(t(\tau),s)}{-(\vol\p_t\vol)(\eta(t(\tau)))}.$$

Then the pair $(\xi,\sigma)$ solves the equation
\be \label{eq:gf3}
\p_\tau\xi=\xi+\sigma \vec H (\xi)+ d \xi_\cdot \grad^{\xi^\ast \cdot}\sigma,\quad \xi\in \tilde{\mathcal A}_k.
\ee Employing the characterization of the orthogonal projection $\tilde P_\xi:T_\xi \mathcal K_k\to T_\xi \tilde{\mathcal A}_k$, we immediately rewrite \eqref{eq:gf3} as a positive gradient flow of the $L^2$-mass: \be\partial_\tau \xi=\tilde P_\xi \xi=\tilde P_\xi \grad_{\mathcal{K}_k} M(\xi) = \grad_{\tilde{\mathcal A}_k} M(\xi). \label{eq:gf4}\ee

\subsection{Evolution of the averaged Lagrange multiplier} In order to illustrate the power of the gradient flow structure \eqref{eq:gf4}, we will formally derive a neat formula for the evolution of the mean of $\sigma$ along the UCMCF trajectories, thereby generalizing (ii) in Proposition \ref{prop:mono}. We first observe that the geodesics in $\tilde{\mathcal A}_k$ are determined by the condition $\p_{\tau\tau}\gamma \perp T_{\gamma}\tilde{\mathcal A}_k $, which can be expressed as \be \label{eq:geoda}
\p_{\tau\tau}\gamma=\varsigma \vec H (\gamma)+ d \gamma_\cdot \grad^{\gamma^\ast \cdot}\varsigma,\quad \gamma\in \tilde{\mathcal A}_k,
\ee
cf. \cite{BMM16,Mol17}. Then we can calculate the Hessian of the $L^2$-mass, taking into account \eqref{key} (with $w=\gamma$ and $\di_{\gamma(\mathcal{M})}(\gamma\circ\gamma^{-1})=k$):
\begin{equation} \label{Hes}
\begin{split}
\langle Hess\; M(\gamma) \dot\gamma, \dot\gamma \rangle_{T_{\gamma}\tilde{\mathcal A}_k}=\frac {d^2 M (\gamma(t))} {d^2t}&=\int_{\mathcal M} \p_{\tau\tau}\gamma\cdot\gamma +\p_\tau\gamma\cdot\p_\tau\gamma\, d \mathcal H^k\\
&=-k\int_{\mathcal M} \varsigma \, d \mathcal H^k+\langle \dot\gamma, \dot\gamma \rangle_{T_{\gamma}\tilde{\mathcal A}_k}.
\end{split}
\end{equation}
Now we compute in two different ways the second time derivative of the $L^2$-mass along a trajectory $\xi(t)$ of the gradient flow \eqref{eq:gf4}. On one hand,
\begin{equation} \label{f33}
\begin{split}
\frac {d^2 M (\xi)} {d^2t}&=\frac d {dt} \langle \grad_{\tilde{\mathcal A}_k} M(\xi), \grad_{\tilde{\mathcal A}_k} M(\xi)
\rangle_{T_{\xi}\tilde{\mathcal A}_k}=2\langle Hess\; M(\xi) \dot\xi, \dot\xi \rangle_{T_{\xi}\tilde{\mathcal A}_k}\\
&=-2k\int_{\mathcal M} \varsigma \, d \mathcal H^k
+2 \langle \dot\xi, \dot\xi
\rangle_{T_{\xi}\tilde{\mathcal A}_k}\\
&=-2k\int_{\mathcal M}
\varsigma \, d \mathcal H^k+2 \langle \grad_{\tilde{\mathcal A}_k} M(\xi), \grad_{\tilde{\mathcal A}_k} M(\xi) \rangle_{T_{\xi}\tilde{\mathcal A}_k}\\
&= -2k\int_{\mathcal M} \varsigma \, d \mathcal H^k+2 \frac {dM(\xi)} {dt},
\end{split}
\end{equation}
where $\varsigma(t)$ is the Lagrange multiplier (which may be referred to as the tension) corresponding to the geodesic passing through $\xi(t)$ at the direction $\dot\xi(t)$, see \eqref{eq:geoda}. On the other hand, employing \eqref{key} and orthogonality of the projection $\tilde P_\xi$, we find that
\begin{equation} \label{f34}
\begin{split}
\langle \grad_{\tilde{\mathcal A}_k} M(\xi), \grad_{\tilde{\mathcal A}_k} M(\xi) \rangle_{T_{\xi}\tilde{\mathcal A}_k} =\langle \tilde P_\xi \xi, \tilde P_\xi  \xi \rangle_{L^2(\mathcal M)}=\langle \xi, \xi \rangle_{L^2(\mathcal M)}+\langle  \xi, \tilde P_\xi  \xi- \xi\rangle_{L^2(\mathcal M)} \\ = \langle \xi, \xi \rangle_{L^2(\mathcal M)}+\left\langle  \xi,  \sigma \vec H (\xi)+ d \xi_\cdot \grad^{\xi^\ast \cdot}\sigma\right\rangle_{L^2(\mathcal M)}=2M(\xi)-k\int_{\mathcal M} \sigma \,d\mathcal H^k.
\end{split}
\end{equation}
This yields the upper bound
\be k\int_{\mathcal M} \sigma \,d\mathcal H^k\leq 2M(\xi).\ee
Differentiating \eqref{f34} in time and comparing with \eqref{f33}, we deduce \be\frac d {dt} \int_{\mathcal M} \sigma \, d \mathcal H^k= 2\int_{\mathcal M} \varsigma \, d \mathcal H^k.\ee In the particular case $\mathcal M=\mathcal S^1$, it is known \cite{Preston-2011,SV17} that the Lagrange multipliers $\varsigma$ related to the geodesics are always non-negative, so we infer (ii) in Proposition \ref{prop:mono}.

\subsection{Relation with the optimal transport} \label{opt.transport} We now establish a link with the  optimal transport theory \cite{villani03topics,villani08oldnew} by explaining how our flow \eqref{eq:gf4} may be formally viewed as a gradient flow on a submanifold of the Wasserstein space. We recall \cite{otto01,villani03topics} that the space $\mathcal P_2(\R^d)$ of probability measures with finite second moments admits a formal Riemannian structure so that the $2$-Wasserstein distance coincides with the geodesic distance. The mapping $$\Pi: \tilde{\mathcal A}_k\to \mathcal P_2(\R^d);\quad \Pi(\xi)=\xi_\#(\mathcal H^{k}\mres \mathcal M)=\mathcal H^{k}\mres \xi(\mathcal M)$$ is a restriction of Otto's Riemannian submersion \cite{otto01}. We refer to \cite{Kli82} for a basic introduction to Riemannian submersions.  Actually, $\Pi$ is a bijection up to volume-preserving diffeomorphisms on $\mathcal{M}$. Set $\rho=\Pi(\xi)$, where $\xi=\xi(\tau)$ is a solution to the normalized UCMCF. Since the $L^2$-mass functional $M$ is invariant with respect to volume-preserving changes of variables on $\mathcal M$, it is consistent to define $M(\rho):=M(\xi)$. Set $\mathcal B=\Pi(\tilde{\mathcal A}_k)$. Then $\mathcal B$ may be formally viewed as a submanifold of $\mathcal P_2(\R^d)$. We claim that the evolution of $\rho$ satisfies \be\partial_\tau \rho= \grad_{\mathcal B} M(\rho). \label{eq:gf5}\ee Indeed, fix time $\tau_0$, $\xi=\xi(\tau_0)$, and $\rho=\Pi(\xi(\tau_0))$. We need to show that \be\langle\partial_\tau \rho,v\rangle_{T_\rho {\mathcal P_2}}= \langle\grad_{\mathcal B} M(\rho),v\rangle_{T_\rho {\mathcal P_2}} \ee for an arbitrary $v\in T_\rho {\mathcal B}$. Let $\tilde\rho(\tau)$ be a curve in $\mathcal B$ satisfying $\tilde\rho(\tau_0)=\rho$, $\p_\tau\tilde\rho(\tau_0)=v$. Let $\tilde\xi$ be the horizontal lift of the curve $\tilde\rho$ with respect to the submersion $\Pi$ passing through $\xi$ at $\tau_0$. Denote by $\p_\tau \xi_h$ the horizontal component of $\p_\tau \xi$. Then at time $\tau_0$ \begin{multline*}\langle\partial_\tau \rho,v\rangle_{\rho}=\langle\partial_\tau \Pi(\xi),\partial_\tau \Pi(\tilde\xi)\rangle_{\rho}=\langle d\Pi(\xi)\cdot \p_\tau \xi,d\Pi(\xi)\cdot \p_\tau \tilde\xi\rangle_{\rho}= \langle \p_\tau \xi_h,\p_\tau \tilde\xi\rangle_{\xi}=\langle \p_\tau \xi,\p_\tau \tilde\xi\rangle_{\xi}\\=\langle \grad_{\tilde{\mathcal A}_k} M(\xi),\p_\tau \tilde\xi\rangle_{\xi}=\p_\tau M(\tilde\xi)=\p_\tau M(\tilde\rho)=\langle \grad_{\mathcal B} M(\rho),\p_\tau \tilde\rho\rangle_{\rho}=\langle\grad_{\mathcal B} M(\rho),v\rangle_{\rho}. \end{multline*} Note that the geodesic distance on $\mathcal B$ is a priori non-degenerate since it is controlled from below by the $2$-Wasserstein distance. The strategy above is applicable to the unnormalized flow. Indeed, the continuation of $\Pi$ defined by $$\Pi: \mathcal A_k^*\to \mathcal P_2(\R^d);\quad \Pi(\eta)=\eta_\#(\mathcal H^{k}\mres \mathcal M)$$ is still a restricted Otto's submersion. Then $\mathcal B^*=\Pi(\mathcal A_k^*)$ may be formally viewed as a submanifold of $\mathcal P_2(\R^d)$, and we are allowed to set $\vol(\rho):=\vol(\eta)$ for $\rho=\Pi(\eta)\in \mathcal B^*$. Then \eqref{eq:gfr} can be recast as \be\partial_t \rho=- \grad_{\mathcal B^*} \ln(\vol(\rho)). \label{eq:gf55}\ee

\subsection*{Acknowledgment}
The first author would like to thank Herbert Koch for helpful discussions related to local and global well-posedness issues. The research was partially supported by CMUC -- UID/MAT/00324/2013, funded by the Portuguese Government through FCT/MCTES and by the ERDF through PT2020, and by the FCT (TUBITAK/0005/2014).

\subsection*{Conflict of interest statement}

We have no conflict of interest to declare.

\end{document}